\documentclass[12pt,oneside]{amsart}

\usepackage{geometry}   
\usepackage{todonotes}
\usepackage{color}
\usepackage{mathrsfs}
\usepackage{bm}
\usepackage{bbm}
\usepackage{wasysym}
\usepackage[utf8]{inputenc}
\inputencoding{utf8}
\usepackage{comment}
\usepackage{graphicx}
\usepackage[justification=centering]{caption}
\usepackage{amsthm,amssymb,amsmath}
\usepackage{subfigure}
\allowdisplaybreaks[4]

\newtheorem{conjecture}{Conjecture}
\newtheorem{question}[conjecture]{Question}
\newtheorem{theorem}{Theorem}[section]
\newtheorem{lemma}[theorem]{Lemma}
\newtheorem{proposition}[theorem]{Proposition}
\newtheorem{corollary}[theorem]{Corollary}

\theoremstyle{definition}

\newtheorem{remarkx}[theorem]{Remark}

\newtheorem{claimx}{Claim}
\newenvironment{claim}
  {\pushQED{\qed}\claimx}
  {\popQED\endclaimx}
\def\NN{\mathbb{N}}
\def\CC{\mathbb{C}}
\def\C{\mathcal{C}}
\def\D{\mathcal{D}}
\def\N{\mathcal{N}}
\def\MM{\mathcal{M}}

\def\RR{\mathbb{R}}
\def\TT{\mathbb{T}}
\def\ZZ{\mathbb{Z}}
\def\e{\mathbbm{1}}
\def\ee{\mathbbm{1}_{\Omega_t}}
\def\l{\ell}

\def\M{\mathscr{M}}

\newcommand{\al}{\alpha}
\newcommand{\be}{\beta}

\newcommand{\ch}{\mathcal H}

\newcommand{\cn}{\mathcal N}

\newcommand{\wt}{\widetilde}
\newcommand{\wh}{\widehat}

\newcommand{\ZR}{\mathbb{R}}
\newcommand{\ZT}{\mathbb{T}}

\newcommand{\ZC}{\mathbb{C}}

\newcommand{\ZN}{\mathbb{N}}

\newcommand{\ti}{\tilde}

\title{The largest $(k,\ell)$-sum-free subsets}

\author{Yifan Jing}
\address{Department of Mathematics, University of Illinois at Urbana-Champaign, Urbana IL, USA}
\email{yifanjing17@gmail.com}
\author{Shukun Wu}
\address{Department of Mathematics, University of Illinois at Urbana-Champaign, Urbana IL, USA}
\email{shukunwu2@illinois.edu}

\date{}
\begin{document}

\dedicatory{Dedicated to the memory of Jean Bourgain}

\begin{abstract}
Let $\mathscr{M}_{(2,1)}(N)$ be the infimum of the largest sum-free subset of any set of $N$ positive integers. An old conjecture in additive combinatorics asserts that there is a constant $c=c(2,1)$ and a function $\omega(N)\to\infty$ as $N\to\infty$, such that $cN+\omega(N)<\mathscr{M}_{(2,1)}(N)<(c+o(1))N$. The constant $c(2,1)$ is determined by Eberhard, Green, and Manners, while the existence of $\omega(N)$ is still wide open.

In this paper, we study the analogous conjecture on $(k,\l)$-sum-free sets and restricted $(k,\l)$-sum-free sets. We determine the constant $c(k,\l)$ for every $(k,\l)$-sum-free sets, and confirm the conjecture for infinitely many $(k,\l)$. 
\end{abstract}

\maketitle

\section{Introduction}
In 1965, Erd\H os asked the following question \cite{Erdos65}. Given an arbitrary sequence $A$ of $N$ different positive integers, what is the size of the largest sum-free subsequence of $A$? By \emph{sum-free} we mean that if $x,y,z\in A$, then $x+y\neq z$. Let
\[
\M_{(2,1)}(N)=\inf_{\substack{A\subseteq\NN^{>0}\\|A|=N}}\max
_{\substack{S\subseteq A\\S\text{ is sum-free}}}|S|.
\]
Using a beautiful probabilistic argument, Erd\H os showed that every $N$-element set $A\subseteq \NN^{>0}$ contains a sum-free subset of size at least $N/3$, in other words, $\M_{(2,1)}(N)\geq N/3$. It turns out that it is surprisingly hard to improve upon this bound. The result was later improved by Alon and Kleitman \cite{AK90}, who showed that $\M_{(2,1)}(N)\geq (N+1)/3$. Bourgain \cite{Bourgain97}, using an entirely different Fourier analytic argument, showed that $\M_{(2,1)}(N)\geq(N+2)/3$, which is the best lower bound on $\M_{(2,1)}(N)$ to date. In particular, the following conjecture has been made in a series of papers. See \cite{Erdos65, Bourgain97, EGM, TV17} for example.
\begin{conjecture}\label{conj:main}
There is a function $\omega(N)\to\infty$ as $N\to\infty$, such that $$\M_{(2,1)}(N)> \frac{N}{3}+\omega(N).$$
\end{conjecture}

On the other hand, a recent breakthrough by Eberhard, Green, and Manners \cite{EGM} proved that $\M_{(2,1)}(N)= (1/3+o(1))N$. More precisely, they showed that for every $\varepsilon>0$, when $N$ is large enough, there is a set $A\subseteq \NN^{>0}$ of size $N$, such that every subset of $A$ of size at least $(1/3+\varepsilon)N$ contains $x,y,z$ with $x+y=z$. This result is one of the first beautiful applications of the arithmetic regularity lemma. Later, using a completely different argument,  the result is generalized by Eberhard \cite{Eberhard15} to \emph{$k$-sum-free} set. A set $A$ is $k$-sum-free if for every $y,x_1,\dots,x_k\in A$, $y\neq \sum_{i=1}^k x_i$. Eberhard proved that  for every $\varepsilon>0$, there is a set $A\subseteq \NN^{>0}$ of size $N$, such that every subset of $A$ of size at least $(1/(k+1)+\varepsilon)N$ contains a $k$-sum.
 For more background we refer to the survey \cite{TV17}.

In this paper, we study the analogue of the Erd\H os sum-free set problem for $(k,\l)$-sum-free sets.  Given two positive integers $k,\l$ with $k>\l$, a set $A$ is \emph{$(k,\l)$-sum-free} if for every $x_1,\dots,x_{k},y_1,\dots,y_\l\in A$, $\sum_{i=1}^kx_k\neq\sum_{j=1}^\l y_j$.
For example, using the notation of $(k,\l)$-sum-free, sum-free is $(2,1)$-sum-free; $k$-sum-free is $(k,1)$-sum-free. Finding largest $(k,\l)$-sum-free sets in some given structures is well-studied in the past fifty years, for example, the size of the maximum $(k,\l)$-sum-free sets in finite cyclic groups was determined recently by Bajnok and Matzke \cite{BM19}, and the size in compact abelian groups was determined by Kravitz \cite{N19}.

For every $A\subseteq\NN^{>0}$, let
\[
\M_{(k,\l)}(A)=\max
_{\substack{S\subseteq A\\S\text{ is } (k,\l)\text{-sum-free}}}|S|,\quad\text{and}\quad \M_{(k,\l)}(N)=\inf_{\substack{A\subseteq\NN^{>0}\\|A|=N}}\M_{(k,\l)}(A).
\]
The problem of determining $\M_{(k,\l)}(N)$ is suggested by Bajnok~\cite[Problem G.41]{Bbook}. In fact, we can also make the following conjecture for $(k,\l)$-sum-free set, which is an analogue of Conjecture \ref{conj:main}.
\begin{conjecture}\label{conj:kl}
Let $k>\l>0$. There is a constant $c=c(k,\l)>0$, and a function $\omega(N)\to\infty$ as $N\to\infty$, such that $$ cN+\omega(N)< \M_{(k,\l)}(N)< (c+\varepsilon)N,$$
for every $\varepsilon>0$.
\end{conjecture}
As we mentioned above, the constant $c(k,\l)$ in Conjecture~\ref{conj:kl} for $(k,\l)=(2,1)$ is determined by Eberhard, Green, and Manners \cite{EGM}, and for $(k,\l)=(k,1)$ is determined by Eberhard \cite{Eberhard15}. The conjecture for $(k,\l)=(3,1)$ is confirmed by Bourgain~\cite{Bourgain97}.

Our first result determines the constant $c(k,\l)$ in Conjecture~\ref{conj:kl} for every $(k,\l)$ (see statements (i) and (iv) of Theorem~\ref{thm:one}), which answers a question asked by Bajnok~\cite{Bbook} when the ambient group is $\ZZ$. The statement (ii) of Theorem~\ref{thm:one} also confirms Conjecture~\ref{conj:kl} for infinitely many $(k,\l)$.
\begin{theorem}\label{thm:one}
Let $k,\l$ be two positive integers and $k>\l$. Then the following hold:\medskip

\noindent\emph{(i)} for every $k,\l$, we have $\M_{(k,\l)}(N)\geq\frac{N}{k+\l}$.\medskip

\noindent\emph{(ii)} suppose $k=5\l$. Then 
\begin{equation}\label{eq: main thm}
\M_{(k,\l)}(N)\geq\frac{N}{k+\l}+c\frac{\log N}{\log\log N},
\end{equation}
\quad\ \  where $c>0$ is an absolute constant that only depends on $k,\l$.\medskip

\noindent\emph{(iii)} for every set $A$ of $N$ positive integers, for every positive even integer $u$, there is an odd integer $v<u$ such that if $k=(u+v)\l/(u-v)$, then
\begin{equation}\label{eq: main thm 2}
\M_{(k,\l)}(A)\geq\frac{N}{k+\l}+c\frac{\log N}{\log\log N},
\end{equation}
\quad\ \  where $c>0$ is an absolute constant that only depends on $k,\l$.\medskip

\noindent\emph{(iv)} for every $k,\l$, we have $\M_{(k,\l)}(N)=\big(\frac{1}{k+\l}+o(1)\big)N$.
\end{theorem}

We remark that Theorem~\ref{thm:one} (iii) also implies estimate (\ref{eq: main thm}) when $k=3\l$, which in particular covers the $(3,1)$-sum-free case obtained by Bourgain. This is because when $u=2$, the only possible value of $v$ is $1$, and this gives us $k=3\l$. It follows that estimate \eqref{eq: main thm 2} holds for every  $N$-element set $A$ when $k=3\l$. Hence, by the definition of $\M_{(k,\l)}(N)$, we prove estimate \eqref{eq: main thm} when $k=3\l$. 


The upper bound construction given by Eberhard, Green, and Manners \cite{EGM} for  $(2,1)$-sum-free set actually works in a more general setting: restricted $(2,1)$-sum-free set. A set $A$ is \emph{restricted $(k,\l)$-sum-free} if for every $k$ distinct elements $a_1,\dots,a_k$ in $A$, and $\l$ distinct elements $b_1,\dots,b_\l$ in $A$, we have $\sum_{i=1}^ka_i\neq\sum_{j=1}^\l b_j$. Let
\[
\widehat{\M}_{(k,\l)}(N)=\inf_{\substack{A\subseteq\NN^{>0}\\|A|=N}}\max
_{\substack{S\subseteq A\\S\text{ is restricted } (k,\l)-\text{sum free}}}|S|.
\]
Clearly, we have that $\M_{(k,\l)}(N)\leq \widehat{\M}_{(k,\l)}(N)$. Our next theorem gives us an upper bound on $\widehat{\M}_{(k,\l)}(N)$ when $k\leq2\l+1$. 

\begin{theorem}\label{thm:1.2}
Let $k,\l$ be positive integers, and $k\leq2\l+1$. Then $$\widehat{\M}_{(k,\l)}(N)=\Big(\frac{1}{k+\l}+o(1)\Big)N.$$
\end{theorem}
\subsection*{Overview}
The paper is organized as follows. In the next section, we provide  some basic definitions and properties in additive combinatorics, harmonic analysis, and model theory (or more precisely, nonstandard analysis) used later in the proof. In Section 3, we prove a variant of the weak Littlewood conjecture, based on the ideas introduced by Bourgain~\cite{Bourgain97}. Theorem~\ref{thm:one} (i) is proved by using the probabilistic argument introduced by Erd\H os, and some structural results for the $(k,\l)$-sum-free open set on the torus. This is included in Section 4. One of the main parts of the paper is to prove Theorem~\ref{thm:one} (ii) and (iii). The special case for $(3,1)$-sum-free set is proved by Bourgain \cite{Bourgain97}, but his argument relies heavily on the fact that a certain term of the Fourier coefficient of the characteristic function is multiplicative, which is not true for the other $(k,\l)$. Here we introduce a different sieve function, as well as a finer control on the functions we constructed. We will discuss it in detail in Section 5. In Sections 6 and 7, we prove Theorem~\ref{thm:one} (iv). The proof goes by showing that the constructions given by Eberhard~\cite{Eberhard15} for $(k,1)$-sum-free sets, the F\o lner sequence, is still the correct construction for the other $(k,\l)$-sum-free sets. The new ingredients contain structural results for the large infinite $(k,\l)$-sum-free sets, which can be viewed as a generalization of the \L uczak–Schoen Theorem \cite{LS97}.
We will prove Theorem~\ref{thm:1.2} in Section 8. In Section~9, we make some concluding remarks, and pose some open problems.

\section{Preliminaries}
\subsection{Additive combinatorics}
Throughout the paper, we use standard definitions and notation in additive combinatorics as given in \cite{additive}. Let $p$ be a prime, and let $m,n,N$ ranging over positive integers. Given $a,b,N\in \NN$ and $a<b$, let $[a,b]:=[a,b]\cap\NN$, and let $[N]:=[1,N]$. We use the standard Vinogradov notation. That is, $f\ll g$ means $f=O(g)$, and $f\asymp g$ if $f\ll g$ and $f\gg g$. 
Given $A,B\subseteq \ZZ$, we write
\[
A+B:=\{a+b\mid a\in A,b\in B\}, \quad\text{and}\quad AB:=\{ab\mid a\in A,b\in B\}.
\]
When $A=\{x\}$, we simply write $x+B:=\{x\}+B$ and $x\cdot B:=\{x\}B$. Given $A\subseteq \ZZ$, let 
\[
kA:=\{a_1+\dots +a_k\mid a_1,\dots,a_k\in A\},
\]
for integer $k\geq2$. For example, $2\cdot\NN$ denotes the set of even natural numbers, while $2\NN$ denotes $\NN+\NN$ which is still $\NN$. Using this notation, a set $A$ is \emph{$(k,\l)$-sum-free} if $kA\cap \l A=\varnothing$.

We also define the restricted sums. Let
\begin{align*}
&A\widehat{+}B:=\{a+b\mid a\in A,b\in B, a\neq b\},\\
&\widehat{kA}:=\{a_1+\dots +a_k\mid a_1,\dots,a_k\in A, \text{ all of them are distinct}\}.
\end{align*}
Thus a set $A$ is \emph{restricted $(k,\l)$-sum-free} if $\widehat{kA}\cap \widehat{\l A}=\varnothing$.

Let $f:\ZZ\to\mathbb{C}$ be a function. Define $\widehat{f}:\TT\to\mathbb{C}$, where $\TT=\RR/\ZZ$ is the 1-dimensional torus, and for every $r\in\TT$,
\[
\widehat{f}(r)=\sum_{x}f(x)e(-rx),
\]
where $e(\theta)=e^{2\pi i\theta}$. By Fourier Inversion, for every $x\in\ZZ$,
\[
f(x)=\int_\TT \widehat{f}(r)e(rx) dr.
\]

Let $\mu:\NN^{>0}\to\mathbb{C}$ be the \emph{M\"obius function}. Recall that $\mu$ is supported on the square-free integers, and $\mu(n)=(-1)^{\omega(n)}$ when $n$ is square-free, where $\omega(n)$ counts the number of distinct prime factors of $n$. By Inclusive-Exclusive Principle, \[
\sum_{d\mid n}\mu(d)=\begin{cases}
0\quad &\text{ if }n>1,\\
1&\text{ if }n=1.
\end{cases}
\]

\subsection{Nonstandard analysis}
We give some basic definitions in nonstandard analysis which will be used later in the proofs. For more systematic accounts we refer to \cite{BT14,non}. Let $S$ be a set with infinitely many elements. An \emph{ultrafilter} $\mathscr{U}$ on $S$ is a collection of subsets of $S$, such that the characteristic function $\mathbbm{1}_{\mathscr{U}}:2^S\to \{0,1\}$ is a finitely additive $\{0,1\}$-valued probability measure on $S$. An ultrafilter is \emph{principal} if it consists of all sets containing some element $s\in S$. Let $\beta S$ denotes the collection of all ultrafilters. One can embed $S$ into $\beta S$, by mapping $x\in S$ to the principal ultrafilter generated by $x$. By a standard application of Zorn's Lemma, $\beta S\setminus S$ is non-empty.

Fix $\mathscr{U}\in \beta \NN\setminus \NN$, and let $M_n$ be a structure for each $n\in \NN$. The \emph{ultraproduct} $\prod_{n\to\mathscr{U}}M_n$ is a space consists of all ultralimits $\lim_{n\to\mathscr{U}}x_n$ of sequences $x_n$ defined in $M_n$, with $\lim_{n\to\mathscr{U}}x_n=\lim_{n\to\mathscr{U}}y_n$ if two sequences $\{x_n\}$ and $\{y_n\}$ agree on a set in $\mathscr{U}$. Let ${^{*}\RR}:=\prod_{n\to \mathscr{U}}\RR$ be the  hyperreal field. Every finite hyperreal number $\xi\in{^{*}\RR}$ is infinitely close to a unique real number $r\in\RR$, called the \emph{standard part} of $\xi$. In this case, we use the notation $r=\mathrm{st}(\xi)$. 

Given a sequence of finite non-empty sets $F_n$, let $\mu_n(X)=|X\cap F_n|/|F_n|$ be a uniform probability measure. Let $F=\prod_{n\to\mathscr{U}}F_n$ be an ultraproduct. The \emph{Loeb measure} \cite{Loeb} $\mu_L$ on $F$ is the unique probability measure on the $\sigma$-algebra generated by the Boolean algebra of internal subsets of $F$, such that when $X=\prod_{n\to\mathscr{U}}X_n$ is an internal subset of $F$, we have
\[
\mu_L(X)=\mathrm{st}\Big(\lim_{n\to\mathscr{U}}\mu_{n}(X_n)\Big).
\]

\subsection{Determinants of certain matrices}
We make use of the following lemma several times in the later proofs, which records a fact about two special matrices. 
\begin{lemma}\label{lem:matrix}
Let $\theta_1,\ldots,\theta_n\in\RR$. Consider two matrices
\begin{equation*}
    A_n=\left(\begin{array}{cccc}
        \sin\theta_1 & \sin\theta_2 & \cdots & \sin\theta_n \\
        \sin2\theta_1 & \sin2\theta_2 & \cdots & \sin2\theta_n \\
        \cdots&\cdots&\cdots&\cdots \\
        \sin n\theta_1 & \sin n\theta_2 & \cdots & \sin n\theta_n
    \end{array}\right),
\end{equation*}
and
\begin{equation*}
    B_n=\left(\begin{array}{cccc}
        1 & 1 & \cdots & 1 \\
        \cos \theta_1 & \cos\theta_2 & \cdots & \cos\theta_n \\
        \cos 2\theta_1 & \cos2\theta_2 & \cdots & \cos2\theta_n \\
        \cdots&\cdots&\cdots&\cdots \\
        \cos (n\!-\!1)\theta_1 & \cos (n\!-\!1)\theta_2 & \cdots & \cos(n\!-\!1)\theta_n
    \end{array}\right).
\end{equation*}
Then we have the formula:
\begin{equation}
\label{determinant-eq-1}
    \det(A_n)=2^{n-1}\big(\prod_{k=1}^n\sin\theta_k\big)\det(B_n);
\end{equation}
and
\begin{equation}
\label{determinant-eq-2}
    \det(B_n)=2^{(n-1)(n-2)/2}\prod_{1\leq k<l\leq n}(\cos\theta_l-\cos\theta_k).
\end{equation}
As a result, 
\begin{equation*}
    \det(A_n)=2^{n(n-1)/2}\big(\prod_{k=1}^n\sin\theta_k\big)\prod_{1\leq k<l\leq n}(\cos\theta_l-\cos\theta_k).
\end{equation*}
\end{lemma}
\begin{proof} For $k=1,2,\ldots n-1$, we subtract the $k$-th row from the $(k+1)$-th row in $A_n$, and use the basic trigonometric identities so that
\begin{eqnarray*}
    \det(A_n)&=&\det\left(\begin{array}{ccc}
        2\sin\frac{\theta_1}{2}\cos\frac{\theta_1}{2}& \cdots & 2\sin\frac{\theta_n}{2}\cos\frac{\theta_n}{2}\\
        \cdots&\cdots&\cdots \\
        2\sin\frac{\theta_1}{2}\cos\frac{(2n\!-\!1)\theta_1}{2} &  \cdots &  2\sin\frac{\theta_n}{2}\cos\frac{(2n\!-\!1)\theta_n}{2}
    \end{array}\right)\\  \nonumber
    &=&2^n\big(\prod_{k=1}^n\sin\frac{\theta_k}{2}\big)\det\left(\begin{array}{ccc}
        \cos\frac{\theta_1}{2}& \cdots & \cos\frac{\theta_n}{2}\\
        \cdots&\cdots&\cdots \\ 
        \cos\frac{(2n\!-\!1)\theta_1}{2} &  \cdots &  \cos\frac{(2n\!-\!1)\theta_n}{2}
    \end{array}\right)\\ \nonumber
    &=:& C_n \det(B_n').
\end{eqnarray*}
For $k=1,2,\ldots n-1$, we add the $k$-th row to the $k+1$-th row in $B_n'$, and use the basic trigonometric identities again so that
\begin{eqnarray*}
    \det(B_n')&=&\det\left(\begin{array}{ccc}
        \cos\frac{\theta_1}{2}& \cdots & \cos\frac{\theta_n}{2}\\
        2\cos\frac{\theta_1}{2}\cos\theta_1 & \cdots & 2\cos\frac{\theta_n}{2}\cos\theta_n\\
        \cdots&\cdots&\cdots \\ 
        2\cos\frac{\theta_1}{2}\cos(n\!-\!1)\theta_1 &  \cdots &  2\cos\frac{\theta_n}{2}\cos(n\!-\!1)\theta_n
    \end{array}\right)\\
    &=&2^{n-1}\big(\prod_{k=1}^n\cos\frac{\theta_k}{2}\big)\det(B_n)
\end{eqnarray*}
Combining the calculations above we prove \eqref{determinant-eq-1}.

As for \eqref{determinant-eq-2}, we let $T_n$ be the Chebyshev polynomial
\begin{equation*}
    T_n(x)=\sum_{k=0}^{\lfloor{n/2}\rfloor}\binom{n}{2k}(x^2-1)^kx^{n-2k}.
\end{equation*}
Thus, we have $T_n(\cos x)=\cos nx$. The coefficient of the leading term, $x^n$ in $T_n(x)$ would be $a_n=2^{n-1}$. Combining this fact and several elementary row operations, we get
\begin{eqnarray*}
    \det(B_n)&=& 2^{(n-1)(n-2)/2}\det\left(\begin{array}{ccc}
        1& \cdots & 1\\
        \cos\theta_1 & \cdots & \cos\theta_n\\
        \cdots&\cdots&\cdots \\ 
        (\cos\theta_1)^{n-1} &  \cdots &  (\cos\theta_n)^{n-1}
    \end{array}\right)\\
    &=&2^{(n-1)(n-2)/2}\prod_{1\leq k<l\leq n}(\cos\theta_l-\cos\theta_k).
\end{eqnarray*}
The last equation comes from the determinant formula for Vandermonde martix. 
\end{proof}

\section{A variant of the Littlewood conjecture}

The Littlewood problem~\cite{Littlewood} is to ask that, what is 
\[
I(N):=\min_{A\subseteq\mathbb{Z}, |A|=N}\int_{\TT}\Big|\sum_{n\in A}e^{inx}\Big|d\mu(x)?
\]
The strong Littlewood conjecture asserts that the minimum occurs when $A$ is an arithmetic progression. This conjecture is still widely open. However, the weak Littlewood conjecture, $I(N)\gg\log N$, is resolved by McGehee, Pigno, and Smith \cite{MPS}, and independently by Konyagin \cite{K81}. The analogous question in discrete setting is also well studied, we refer to \cite{Green, Sanders, S17} for the interested readers. In this section, we will develop a variant of the weak Littlewood conjecture, based on the ideas given by Bourgain~\cite{Bourgain97}.

Let $\cn$ be the set of natural numbers that only contains prime factors at least $P$, where $P\asymp(\log N)^{100}$ is a prime. We will use the following lemma from \cite[Section~5]{Bourgain97}.
\begin{lemma}
\label{bourgain-lemma}
Let $A$ be a finite subset of $\ZZ^+$ with $|A|=N$. For all $R\geq1$, we define
\begin{equation*}
    A_R=\{m\in A:m<R\}.
\end{equation*}
Also, we use ${\rm Proj}_R\sum a_ke^{ikx}$ to denote the truncated sum $\sum_{|k|\leq R}a_ke^{ikx}$. Assume $|a_n|\leq 1$ and $P>(\log N)^{20}$. Then there is an absolute big constant $C$, such that
\begin{equation*}
    \Bigg\|{\rm Proj}_R\sum_{\substack{n\in\cn, m\in A}}\frac{a_n}{n}e^{imnx}\Bigg\|_2< CP^{-1/15}|A_R|^{1/2}.
\end{equation*}
\end{lemma}
With Lemma~\ref{bourgain-lemma} in hand, we are able to prove our technical lemma:
\begin{lemma}
\label{technical-lemma}
Let $A$ be a finite subset of $\ZN^{>0}$ with $|A|=N$ and let $P>(\log N)^{100}$. Assume $|a_n|\leq 1$. Then for any $r_0\in\ZN^{>0}$, there exists a function $\Phi(x)$ with $\|\Phi\|_\infty<10$ such that 
\begin{equation}
\label{technical-estimate-1}
    \Big|\Big\langle\sum_{m\in A}(e^{imx}+e^{ir_0mx}),\Phi(x)\Big\rangle\Big|\geq c\log N;
\end{equation}
while for any $\be\in \ZZ$,
\begin{equation}
\label{technical-estimate-2}
    \Big|\Big\langle\sum_{\substack{n\in\cn, m\in A}}\frac{a_n}{n} e^{i\beta mnx}, \Phi(x)\Big\rangle\Big|\leq C(\log N)^{-2}.
\end{equation}
Here $c,C$ are two absolute constants.
\end{lemma}
\begin{proof} For convenience, we assume $A=\{m_1,\ldots,m_N\}$, and define
\begin{equation*}
    G(x)=\sum_{j=1}^N e^{im_jx}+e^{ir_0m_jx}.
\end{equation*}

Let $k_0$ be the largest natural number that $10^{6k_0}<N$. We group $A$ into disjoint subsets $\{B_k\}_{k=0}^{k_0}$ such that for $0\leq k\leq k_0-1$, $|B_k|=10^{6k}$. Here $B_0=\{m_1\}$, $B_1=\{m_2,\ldots,m_{10^6+1}\},\cdots$ and $B_{k_0}=A\setminus(\bigcup_{k\leq k_0-1}B_k)$. From the construction we know $|B_{k_0}|\asymp 10^{6k_0}$. For each $B_k$, we define 
\begin{equation*}
    \wt P_{k}=\frac{1}{|B_k|}\sum_{m\in B_k} e^{imx}.
\end{equation*}
We also define, after setting $F_M=\sum_{|m|\leq M}\frac{M-|m|}{M}e^{imx}$ to be the $M$-F\'ejer kernel, 
\begin{equation*}
    P_k=\wt P_{k}\ast (e^{i\xi_k x}F_{|I_k|}).
\end{equation*}
Here $I_k=[a_k,b_k]$ is the interval with $a_k=\min\{m,m\in B_k\}$, $b_k=\max\{m,m\in B_k\}$, and $\xi_k$ is the center of $I_k$. As a result, we have
\begin{equation}
\label{support-P-k}
    {\rm supp}(\wh P_k)={\rm supp}(\wh{\wt P_k})\subset I_k,
\end{equation}
and 
\begin{equation*}
    \langle G,P_k\rangle >\frac{1}{2}.
\end{equation*}

Next, for each $P_k$, we define a function $Q_k$ that served as a ``compensator". Let $\ch$ be the Hilbert transform in $L^2(\ZT)$ that $\wh{\ch{f}}(n)=-i{\rm sgn}(n)\wh f(n)$, so that $\ch f(x)\in\ZR$ when $f$ is a real-valued function. We set   
\begin{equation}
\label{Q-k}
    Q_k=\Big(e^{-(|\ti P_{k}|-i\ch[|\ti P_{k}|])}\Big)\ast F_{|I_k|}.
\end{equation}
Since the Fourier series of $|\ti P_{k}|-i\ch[|\ti P_{k}|]$ is supported in non-positive integers, the Fourier series of $e^{-(|\ti P_{k}|-i\ch[|\ti P_{k}|])}$ has the same support. Hence
\begin{equation}
\label{support-Q-k}
    {\rm supp}(\wh Q_k)\subset[-|I_k|,0].
\end{equation}
Recall the inequality $|e^{-z}-1|\leq |z|$ if $z\in\ZC$ and ${\rm Re}(z)\geq0$. Thus, noticing that the Hilbert transform $\ch$ is an $L^2$ isometry, we have 
\begin{equation*}
    \|1-Q_k\|_2\leq\big\|e^{-(|\ti P_{k}|-i\ch[|\ti P_{k}|])}-1\big\|_2\leq\|\ti P_k\|_2+\|\ch[|\ti P_{k}|]\|_2<2|B_k|^{-1/2}.
\end{equation*}

\vspace{3mm}
We will use the functions $P_k,Q_k$ to construct our test function $\Phi$. Specifically, we let $\Phi_0=P_0$ and let
\begin{equation}
\label{inductive-def}
    \Phi_k=Q_{k}\Phi_{k-1}+P_k, \hspace{1cm}1\leq k\leq k_0.
\end{equation}
Define $\Phi=\Phi_{k_0}$. We can also write down the explicit formula for $\Phi$ by
\begin{equation}
\label{explicit-def}
    \Phi=P_{k_0}+P_{k_0-1}Q_{k_0}+P_{k_0-2}Q_{k_0-1}Q_{k_0}+\cdots+P_0Q_1\cdots Q_{k_0}.
\end{equation}

We claim that $\|\Phi\|_\infty<10$. To see this, we first recall the basic inequality: $\frac{a}{10}+e^{-a}\leq1$ if $a\geq0$. Then, observing that $|P_0|=1$ and
\begin{equation*}
    \Big\|\frac{1}{10}|P_k|+|Q_k|\Big\|_\infty\leq\Big\|\big(\frac{1}{10}|\wt P_k|+e^{-|\wt P_k|}\big)\ast F_{|I_k|}\Big\|_\infty\leq\Big\|\frac{1}{10}|\wt P_k|+e^{-|\wt P_k|}\Big\|_\infty\leq1,
\end{equation*}
we argue inductively using \eqref{inductive-def} to conclude our claim.

\vspace{3mm}
Next, we will verify \eqref{technical-estimate-1}. Write
\begin{equation*}
    G_1(x)=\sum_{j=1}^N e^{im_jx}, \hspace{1cm}G_2(x)=\sum_{j=1}^Ne^{ir_0m_jx}.
\end{equation*}
Also, recalling $b_k=\max\{m,m\in B_k\}$, we define two truncated series
\begin{equation*}
    G_{1,k}(x)=\sum_{m\leq b_k,m\in A} e^{imx}, \hspace{1cm}G_{2, k}(x)=\sum_{m\leq b_k,m\in A}e^{ir_0mx}.
\end{equation*}
Therefore, using \eqref{support-P-k}, \eqref{support-Q-k}, \eqref{explicit-def} and the fact $\wh P_k\geq0$, 
\begin{eqnarray}
    |\langle G,\Phi\rangle| \!\!\!&=&\!\!\! \Big|\sum_{k=0}^{k_0}\langle G,P_{k}\rangle+\sum_{k=0}^{k_0}\langle G_1+G_2,P_k(1-Q_{k+1}\cdots Q_{k_0})\rangle\Big|\\ \label{error-in-second}
    &\geq&\!\!\! \frac{k_0}{2}-\Big|\sum_{k=0}^{k_0}\langle G_{1,k}+G_{2,k},P_k(1-Q_{k+1}\cdots Q_{k_0})\rangle\Big|.
\end{eqnarray}

Observing that $\|P_k\|_\infty,~\|Q_k\|_\infty\leq1$ and
\begin{equation}
\nonumber
    1-Q_{k+1}\cdots Q_{k_0}=(1-Q_{k+1})+Q_{k+1}(1-Q_{k+2})+\cdots+(1-Q_{k_0})Q_{k+1}\ldots Q_{k_0-1},
\end{equation}
we can derive the following estimates for $0\leq k\leq k_0-1$:
\begin{eqnarray}
    &&|\langle G_{1,k},P_k(1-Q_{k+1}\cdots Q_{k_0})\rangle|\leq \|G_{1,k}\|_2\cdot\|P_k(1-Q_{k+1}\cdots Q_M)\|_2\\ \nonumber
    &&\leq 2\times 10^{3(k-1)}\sum_{j=k+1}^{k_0}\|1-Q_j\|_2\leq8\times 10^{-3}.
\end{eqnarray}
The last inequality follows from $\|1-Q_j\|_2\leq2|B_k|^{-1/2}$.

Similarly, we can prove $|\langle G_{2,k},P_k(1-Q_{k+1}\cdots Q_{k_0})\rangle|\leq 8\times 10^{-3}$. Plugging these two estimates back to \eqref{error-in-second}, summing up $k$ and using the triangle inequality so that we can conclude 
\begin{equation*}
    \langle G,\Phi\rangle\geq\frac{k_0}{3}.
\end{equation*}
The desired estimate \eqref{technical-estimate-1} follows readily as $k_0\geq(\log N)/100$.

\vspace{3mm}
Finally, we are going to verify \eqref{technical-estimate-2}. Let 
\begin{equation*}
    H(x)=\sum_{\substack{n\in\cn, m\in A}}\frac{a_n}{n} e^{i\beta mnx}.
\end{equation*}
From \eqref{support-P-k} and \eqref{support-Q-k}, we know that for $k+1\leq j\leq k_0$,
\begin{equation*}
    {\rm supp}\{\big(P_k(1-Q_{k+1}\cdots Q_{j})\big)^\wedge\}\subset[-b_j, b_k].
\end{equation*}
Thus, using \eqref{explicit-def} we have
\begin{equation}
\label{error-estimate}
    \langle H, \Phi\rangle=\sum_{k=0}^{k_0}\langle H,P_{k}\rangle+\sum_{k=0}^{k_0}\sum_{j=k+1}^{k_0}\langle H,P_kQ_{k+1}\cdots Q_{j-1}(1-Q_j)\rangle.
\end{equation}
Here we set $Q_{k+1}Q_k=1$ in convention. Since ${\rm supp}(\wh P_k)\subset[0,b_k]$, we apply Lemma \ref{bourgain-lemma} so that
\begin{equation}
\nonumber
    |\langle H, P_{k}\rangle|\leq\|P_k\|_2\cdot\|{\rm Proj}_{b_k} H\|_2<C10^{-3(k-1)}P^{-1/15}10^{3(k-1)}=CP^{-1/15}.
\end{equation}
Summing up all the $k\leq k_0$ using the triangle inequality, we can bound the first term in \eqref{error-estimate} with
\begin{equation}
\label{first-term}
    \sum_{k=0}^{k_0}|\langle H,P_{k}\rangle|< CP^{-1/15}k_0.
\end{equation}

For the second term in \eqref{error-estimate}, we similarly have
\begin{equation*}
    |\langle H,P_kQ_{k+1}\cdots Q_{k_{j-1}}(1-Q_j)\rangle|\leq\|1-Q_j\|_2\cdot\|{\rm Proj}_{b_j}H\|_2<CP^{-1/15}.
\end{equation*}
Summing up all the $k+1\leq j\leq k_0$ and $0\leq k\leq k_0$ using the triangle inequality again, we therefore can conclude 
\begin{equation}
\label{second-term}
    \sum_{k=0}^{k_0}\sum_{j=k+1}^{k_0}|\langle H,P_kQ_{k+1}\cdots Q_{j-1}(1-Q_j)\rangle|<CP^{-1/15}k_0^2.
\end{equation}

We conclude the proof of \eqref{technical-estimate-2} by the facts $k_0<\log N$ and $P>(\log N)^{100}$. \end{proof}

\begin{remarkx}
The above argument can be easily generalized with \eqref{technical-estimate-1} replaced by the requirement
\begin{equation}
\label{technical-estimate-3}
    \Big|\Big\langle\sum_{m\in A}\sum_{r\in\Lambda}e^{irmx},\Phi(x)\Big\rangle\Big|\geq c\log N.
\end{equation}
Here $\Lambda\subset\ZN^{>0}$, and the constant $c$ only depends on the size of $\Lambda$.
\end{remarkx}

As an application of Lemma \ref{technical-lemma}, we have the following corollary:
\begin{corollary}\label{cor: key}
Let $A$ be a finite subset of $\ZN^{>0}$ with $|A|=N$ and let $P>(\log N)^{100}$. Recall that $\cn$ is the set of natural numbers that only contains prime factors at least $P$. Assume $|a_n|\leq 1$. Then for any $r_0\in\ZN^{>0}$, $\Gamma\subset\ZZ$ with $|\Gamma|\leq \log N$, we have
\begin{equation*}
    \Bigg\|\sum_{m\in A}\Big(e^{imx}+e^{ir_0mx}\Big)+\sum_{\substack{n\in\cn, m\in A}}\Big(\sum_{\beta\in\Gamma}\frac{a_n}{n}e^{i\beta mnx}\Big)\Bigg\|_1\geq c\log N.
\end{equation*}
\end{corollary}
\begin{proof} We apply Lemma \ref{technical-lemma} to obtain a function $\Phi(x)$ satisfying \eqref{technical-estimate-1} and \eqref{technical-estimate-2}. Then we can conclude the corollary as
\begin{eqnarray}
\nonumber
    &&\Bigg\|\sum_{m\in A}\Big(e^{imx}+e^{ir_0mx}\Big)+\sum_{\substack{n\in\cn, m\in A}}\Big(\sum_{\beta\in\Gamma}\frac{a_n}{n}e^{i\beta mnx}\Big)\Bigg\|_1\|\Phi\|_\infty\\ \nonumber
    &&\geq \Big|\Big\langle\sum_{m\in A}(e^{imx}+e^{ir_0mx}),\Phi(x)\Big\rangle\Big|-\sum_{\beta\in\Gamma}\Big|\Big\langle\sum_{\substack{n\in\cn, m\in A}}\frac{a_n}{n} e^{i\beta mnx}, \Phi(x)\Big\rangle\Big|\\ \nonumber
    &&> c\log N.
\end{eqnarray}
\end{proof}

\section{$(k,\l)$-sum-free open sets in the torus}

In this section, we use $\mu_H$ as the Haar probability measure on $\TT$. 
\begin{proposition}\label{prop:T1}
Let $A\subseteq \TT$ be a $(k,\l)$-sum-free open set. Then $\mu_H(A)\leq\frac{1}{k+\ell}$.
\end{proposition}
\begin{proof}
Since $A$ is $(k,\l)$-sum-free, we have $kA\cap \l A=\varnothing$. In particular, $\mu_H(kA)+\mu_H(\l A)\leq 1$. By Kneser's inequality \cite{kenser},
\[
(k+\l)\mu_H(A)\leq \mu_H(kA)+\mu_H(\l A)\leq 1,
\]
which implies that $\mu_H(A)\leq 1/(k+\l)$.
\end{proof}

Next, we construct some largest $(k,\l)$-sum-free open sets in $\TT$. When $k-\l\geq2$, our construction is asymmetric, which will help us get a better lower bound on $\M_{(k,\l)}(N)$. We will discuss this in details in the next section.

\begin{lemma}\label{lem:T}
Let $k,\l$ be two positive integers and $k>\l$.
For every integer $t\in [k-\l]$, set $\Omega_t=\big(\frac{t-1}{k-\l}+\frac{\ell}{k^2-\ell^2},\frac{t-1}{k-\l}+\frac{k}{k^2-\ell^2}\big)$.
Then $\Omega_t$ is $(k,\l)$-sum-free.
\end{lemma}

Lemma~\ref{lem:T} is easy to verify, and we omit the details here.
When $k=\l+1$, the following observation shows that all the possible $(k,\l)$-sum-free open sets with maximum measure are symmetric. Thus one cannot apply the method used in the next section to improve the lower bound for the cases $k=\l+1$.

\begin{lemma}\label{lem: k=l+1}
Let $k=\l+1$. Suppose $A\subseteq \TT$ is a maximum $(k,\l)$-sum-free open set. Then $A$ is symmetric.
\end{lemma}
\begin{proof}
Since $k=\l+1$, $A$ is $(k,\l)$-sum-free implies that $(\l A-\l A)\cap A=\varnothing$. Hence $A\subseteq \TT\setminus (\l A-\l A)$. By Kneser's inequality, 
\[
\mu_H(\TT\setminus (\l A-\l A))\leq 1-2\l\mu_H(A).
\]
By Proposition~\ref{prop:T1}, $\mu_H(A)=\frac{1}{2\l+1}$. Thus $A=\TT\setminus(\l A-\l A)$, and this implies that $A$ is symmetric.
\end{proof}

Using the argument by Erd\H os \cite{Erdos65}, Lemma~\ref{lem:T} is able to give us the following lower bound on the maximum $(k,\l)$-sum-free subsets of any set of $N$ integers, which proves Theorem~\ref{thm:one} (i).
\begin{proposition}
Let $k,\l$ be positive integers and $k>\l$. Then for every $A\subseteq \NN^{>0}$ of size $N$, $A$ contains a $(k,\l)$-sum-free subsets of size at least $\frac{1}{k+\ell}N$.
\end{proposition}
\begin{proof}
Let $\Omega_t$ be as in Lemma~\ref{lem:T}, and let $\e_\Omega$ be the characteristic function of $\Omega$ in $\TT$. Thus by Fubini's Theorem,
\[
\int_\TT\sum_{n\in A}\e_\Omega(nx)d\mu_H(x)=\sum_{n\in A}\int_\TT\e_\Omega(nx)d\mu_H(x)=\frac{N}{k+\ell}.
\]
Therefore, by Pigeonhole principle, there exists $x\in\TT$ such that $$|\{n\in A\mid nx\in \Omega\}|\geq\frac{N}{k+\ell},$$
finishes the proof.
\end{proof}

\section{Lower Bounds}

Let $k,\l$ be two positive integers with $k-\l\geq2$. Let $I=\{1,\dots,k-\l\}$ be the index set. Set
\[
\Omega_t=\Big(\frac{t-1}{k-\l}+\frac{\ell}{k^2-\ell^2},\frac{t-1}{k-\l}+\frac{k}{k^2-\ell^2}\Big),
\]
for every $t\in I$.
Let $\ee$ be the indicator function of $\Omega_t$. Given $A\subseteq \NN^{>0}$ of size $N$, let $\M(A)$ be the size of the maximum $(k,\l)$-sum-free subset of $A$. We have
\begin{align}\label{eq:MA}
    \M(A)\geq\max_{x\in\TT}\sum_{n\in A}\ee(nx), 
\end{align}
since $\Omega_t$ is $(k,\l)$-sum-free for every $t$. Then
\begin{align}\label{eq:main}
  &\max_{x\in\TT}\sum_{n\in A}\e_{\Omega_t}(nx)=\frac{N}{k+\l}+\max_{x\in\TT}\sum_{n\in A}\Big(\e_{\Omega_t}-\frac{1}{k+\l}\Big)(nx),
\end{align}
for every $t\in I$. We introduce a balanced function $f_t:\TT\to\CC$ defined by $f_t=\e_{\Omega_t}-\frac{1}{k+\l}$. By orthogonality of characters we have
\[
 \widehat{f_t}(n)=
 \begin{cases}
 0\quad & \text{ if } n=0,\\
 \widehat{\e_{\Omega_t}}(n) & \text{ else}.
 \end{cases}
\]

By Fourier inversion, when $n>0$,
\begin{align*}
    \widehat{f_t}(n)&=\int_{\TT}\ee(x)e(-nx)d\mu(x)\\
    &=\frac{1}{2\pi in}\Big(-e\big(-\frac{(t-1)n}{k-\l}-\frac{nk}{k^2-\ell^2}\big)+e\big(-\frac{(t-1)n}{k-\l}-\frac{n\l}{k^2-\l^2}\big)\Big).
\end{align*}
Simplify $\wh{f}_t(n)$ as
\begin{align*}
   \widehat{f_t}(n)&=\frac{1}{2\pi n} e\Big(\frac{(2t-1)n}{2(k-\l)}\Big)\Big(\sin(\frac{2k n\pi}{k^2-\l^2}-\frac{\pi n}{k-\l})-\sin(\frac{2\l n\pi}{k^2-\l^2}-\frac{\pi n}{k-\l})\Big)\\
   &=\frac{1}{\pi n}e\Big(\frac{(2t-1)n}{2(k-\l)}\Big)\sin\Big(\frac{n\pi}{k+\l}\Big).
\end{align*}
Hence, for every $t\in I$ we have
\begin{align*}
    &f_t(x)=\sum_{n\neq0}\widehat{f_1}(n)e(nx)=\sum_{n\neq0}\frac{1}{\pi n}e\Big(\frac{(2t-1)n}{2(k-\l)}\Big)\sin\Big(\frac{n\pi}{k+\l}\Big)e(nx).
\end{align*}
Let $F(x):=\sum_{t\in I}f_t(x)$. The sine terms cancel when summing up $t$ as
\[
\sum_{t=1}^{k-\l}\sin\Big(\frac{(2t-1)n\pi}{k-\l}\Big)=0,
\]
so we get
\begin{align*}
    F(x)&=\frac{1}{\pi}\sum_{n\geq1}\frac{1}{n}\sin\Big(\frac{n\pi}{k+\l}\Big)\alpha(n)\big(e(nx)+e(-nx)\big),
\end{align*}
where $\al(n):\ZZ\to\ZC$ is defined by 
\[
    \alpha(n)=\sum_{t\in I}\cos\Big(\frac{(2t-1)n\pi}{k-\l}\Big)=
    \begin{cases}
    0 \quad&\text{ when } (k-\l)\nmid n,\\
    (-1)^{s}(k-\l) &\text{ when } n=(k-\l)s.
    \end{cases}
\]
Therefore, we have
\begin{equation}\label{eq:Fx}
F(x)=\frac{2}{\pi}\sum_{n\geq1}\frac{(-1)^{n}}{n}\sin\Big(\frac{(k-\l)n\pi}{k+\l}\Big)\cos(2\pi(k-\l)n x).
\end{equation}

In the rest of the section, we let $k-\l$ be an even integer. Set
\[
I_1=\{1,\dots,(k-\l)/2\}, \quad I_2=\{(k-\l)/2+1,\dots,k-\l\}.
\]
We define a ($\frac{k-\l}{2}\times \frac{k-\l}{2}$)-matrix $D=(d_{ij})$, such that
\[
d_{ij}=\sin \Big(\frac{i(2j-1)\pi}{k-\l}\Big)
\]
for every $i,j\in I_1$.

Let $\bm{\lambda}=(\lambda_1,\dots,\lambda_{(k-\l)/2})$ be a vector. By Lemma~\ref{lem:matrix}, there is $\bm{\lambda}\in\RR^{(k-\l)/2}$, with $|\lambda_i|\leq k^k$, such that $D\bm{\lambda}^{T}=(0,\dots,0,1)^T$. Fix this $\bm{\lambda}$, and let
\[
G(x)=\sum_{j\in I_1}\lambda_j f_j(x)-\sum_{t\in I_2}\lambda_{k-\l+1-t} f_t(x).
\]
Observe that for any $n\in\ZN^{>0}$,
\[
\sum_{t\in I_1}\lambda_t\cos\Big(\frac{(2t-1)n\pi}{k-\l}\Big)=\sum_{t\in I_2}\lambda_{k-\l+1-t}\cos\Big(\frac{(2t-1)n\pi}{k-\l}\Big).
\]
As a result, we have
\[
G(x)=\frac{1}{\pi}\sum_{n\geq1}\frac{1}{n}\sin\Big(\frac{n\pi}{k+\l}\Big)\beta(n)\sin(2\pi nx),
\]
where
\begin{align*}
     \beta(n)&=\sum_{j\in I_2}\lambda_{k-j}\sin\Big(\frac{(2j-1)n\pi}{k-\l}\Big)-\sum_{t\in I_1}\lambda_t\sin\Big(\frac{(2t-1)n\pi}{k-\l}\Big)\\
     &=
    \begin{cases}
    0 \quad&\text{ when } n\neq\frac{k-\l}{2}(2s-1),\\
    2(-1)^{s+1} &\text{ when } n=\frac{k-\l}{2}(2s-1).
    \end{cases}
\end{align*}
Therefore, we get
\begin{equation}\label{eq:Gx}
G(x)=\frac{2}{\pi(k-\l)}\sum_{n\geq1}\frac{\gamma(n)}{n}\sin\Big(\frac{(k-\l)n\pi}{2(k+\l)}\Big)\sin(\pi(k-\l)nx),
\end{equation}
where $\gamma(n)=\beta((k-\l)n/2).$ We now split the proof into two cases.

\subsection{Proof of Theorem~\ref{thm:one} (ii)} Now we have $k=5\l$. On one hand,
by equation (\ref{eq:Fx}), we have
\[
F(x)=-\frac{\sqrt3}{\pi}\sum_{n\geq1}\frac{\psi(n)}{n}\cos(8\pi\l nx),
\]
where
\[
\psi(n)=
\begin{cases}
1&\text{ when }n\equiv 1,2\pmod6,\\
-1&\text{ when }n\equiv 4,5\pmod6,\\
0&\text{ otherwise}.
\end{cases}
\]
Note that $\psi(n)$ is \emph{not} a multiplicative function. Using the M\"obius function $\mu$, we define a weighted M\"obius function $\eta$ that
\[
\eta(n)=
\begin{cases}
\mu(n)&\text{ when } n\equiv 1,4\pmod 6,\\
-\mu(n)&\text{ when } n\equiv 2,5\pmod 6,\\
0&\text{ otherwise. }
\end{cases}
\]
Set $P\asymp(\log N)^{100}$ a prime. Let $\MM_1$ be the set of square-free integers such that for every $n\in\MM_1$ we have $3\nmid n$, all the prime factors of $n$ are at most $P$, and we further require that $1\in\MM_1$. Then, we have
\begin{align*}
    \sum_{m\in\MM_1}\frac{\eta(m)}{m}F(mx)=-\frac{\sqrt3}{\pi}\sum_{n\geq1}\frac{1}{n}\cos(8\pi\l nx)\sum_{m\in\MM_1, m\mid n}\eta(m)\psi\Big(\frac{n}{m}\Big),
\end{align*}
where
\begin{align*}
    \sum_{m\in\MM_1, m\mid n}\eta(m)\psi\Big(\frac{n}{m}\Big)=\sum_{\substack{m\in\MM_1, m\mid n\\ 2\nmid m}}\eta(m)\psi\Big(\frac{n}{m}\Big)+\sum_{\substack{m\in\MM_1, m\mid n\\ 2\mid m}}\eta(m)\psi\Big(\frac{n}{m}\Big)=: I_1+I_2.
\end{align*}
Note that $I_1+I_2=0$ when $3\mid n$. Also, recall that $\N$ is the set defined in Section 3 that contains integers only having prime factors at least $P$. It follows that for any odd integer $n\not\in\N$ with $3\nmid n$, 
\begin{equation}
    \sum_{\substack{m\in\MM_1, m\mid n}}\eta(m)\psi\Big(\frac{n}{m}\Big)=\psi(n) \sum_{\substack{m\in\MM_1, m\mid n}}\mu(m)=0.
\end{equation}
As a consequence, when $n$ is an odd integer with $n\not\in\N$, we have $I_1=I_2=0$, unless $n=1$; When $n$ is an even integer with $n/2\not\in \N$ and $n\neq2^d$, we have $I_1(n)=0$ and $I_2(n)=I_1(n/2)=0$. When $n=2^d$, we have
\[
I_1+I_2=\psi(2^d)+\psi(2^{d-1})=\begin{cases}
2 &\text{ when } d=1,\\
0 &\text{ when } d>1.
\end{cases}
\]
Therefore,
\begin{eqnarray}
\nonumber
    \sum_{m\in\MM_1}\frac{\eta(m)}{m}F\Big(\frac{mx}{4}\Big) \!\!\!&=&\!\!\! -\frac{\sqrt3}{\pi}\Bigg(\cos(2\pi\l x)+\cos(4\pi\l x)+\\ \nonumber
    &&\sum_{n\in\N}\frac{\psi(n)}{n}\cos(2\pi\l nx)+\sum_{n\in 2\cdot\N}\frac{\psi(n)+\psi(n/2)}{n}\cos(2\pi\l nx)\Bigg),    
\end{eqnarray}

\vspace{3mm}

On the other hand, by equation (\ref{eq:Gx}), we get
\[
G(x)=\frac{\sqrt 3}{\pi}\sum_{n\geq1}\frac{\pi(n)}{n}\sin(4\pi\l nx),
\]
where
\[
\pi(n)=
\begin{cases}
1 &\text{ when } n\equiv \pm 1\pmod{12},\\
-1 &\text{ when } n\equiv \pm 5\pmod{12},\\
0 &\text{ otherwise. }
\end{cases}
\]
Note that $\pi(n)$ is a multiplicative function. Let $\MM$ be the set of square-free integers such that for every $n\in\MM$, all the prime factors of $n$ are at most $P$, and $1\in\MM$. Thus by the basic properties of the M\"obius function, we have
\begin{align*}
    \sum_{m\in\MM}\frac{\mu(m)\pi(m)}{m}G(mx)&=\frac{\sqrt 3}{\pi(k-\l)}\sum_{n\geq1}\frac{\pi(n)}{n}\sin(4\pi\l nx)\sum_{m\in\MM, m\mid n}\mu(m)\\
    &=\frac{\sqrt 3}{\pi(k-\l)}\Bigg(\sin(4\pi\l x)+\sum_{n\in\N}\frac{\pi(n)}{n}\sin(4\pi\l nx)\Bigg).
\end{align*}

Now we are going to apply Corollary~\ref{cor: key} to obtain a lower bound of $\M(A)$. Let
\[
\pi'(n)=\begin{cases}
\pi(n)&\text{ when } 2\nmid n,\\
\pi(n/2)&\text{ when } 2\mid n\text{ and } 4\nmid n,\\
0&\text{ otherwise,}
\end{cases}
\]
and let
\[
\psi'(n)=\begin{cases}
\psi(n)&\text{ when } 2\nmid n,\\
\psi(n)+\psi(n/2)&\text{ when } 2\mid n\text{ and } 4\nmid n,\\
0&\text{ otherwise.}
\end{cases}
\]
We have
\vspace{3mm}
\begin{align*}
&\,\log N\\
\ll&\,\Bigg\Vert \sum_{m\in A}\big(\cos(2\pi\l mx)+\cos(4\pi\l mx)+i\sin(2\pi\l mx)+i\sin(4\pi\l mx)\big)\\
&\,+\sum_{m\in A}\sum_{n\in\N\cup(2\cdot\N)}\frac{1}{n}\Big(\psi'(n)\cos(2\pi\l mnx)+i\pi'(n)\sin(2\pi\l mnx)\Big)\Bigg\Vert_{L^1(\TT)}\\
\ll&\,\Bigg\Vert \sum_{t\in\MM_1}\frac{\eta(t)}{t}\sum_{m\in A}F\Big(\frac{tmx}{4}\Big)\Bigg\Vert_{L^1(\TT)}+\Bigg\Vert \sum_{t\in\MM}\frac{\mu(t)\pi(t)}{t}\sum_{m\in A}G\Big(\frac{tmx}{2}\Big)\Bigg\Vert_{L^1(\TT)}\\
&\,+\Bigg\Vert \sum_{t\in\MM}\frac{\mu(t)\pi(t)}{t}\sum_{m\in A}G(tmx)\Bigg\Vert_{L^1(\TT)}\\ 
\ll&\,\sum_{t\in\MM_1}\bigg|\frac{\eta(t)}{t}\bigg|\bigg\Vert \sum_{m\in A}F(mx)\bigg\Vert_{L^1(\TT)}+2\sum_{t\in\MM}\bigg|\frac{\mu(t)\pi(t)}{t}\bigg| \Bigg\Vert\sum_{m\in A}G(mx)\bigg\Vert_{L^1(\TT)}\\
\ll&\,\prod_{p\leq P}\Big(1+\frac{1}{p}\Big)\bigg(\sum_{t=1}^{k-\l}\bigg\Vert \sum_{m\in A}f_t(mx)\bigg\Vert_{L^1(\TT)}+2\sum_{t\in I_1\cup I_2}\lambda_t\bigg\Vert \sum_{m\in A}f_t(mx)\bigg\Vert_{L^1(\TT)}\bigg).
\end{align*}
By Mertens' estimates we get
\[
\prod_{p\leq P}\Big(1+\frac{1}{p}\Big)\ll\log P\asymp \log\log N.
\]
Hence there is $t\in I$ such that $\big\Vert \sum_{m\in A}f_t(mx)\big\Vert_{L^1(\TT)}\gg\frac{\log N}{\log\log N}$.

Note that,
\[
\int_{\TT}\sum_{n\in A}f_t(nx)dx=0.
\]
Thus we have
\[
\max_{x\in\TT}\sum_{n\in A}f_t(nx)\geq\frac{1}{2}\bigg\Vert \sum_{n\in A}f_t(nx)\bigg\Vert_{L^1(\TT)}.
\]
Together with (\ref{eq:MA}) and (\ref{eq:main}), we get
\[
\M(A)-\frac{N}{k+\l}\gg\frac{\log N}{\log\log N},
\]
and this proves Theorem~\ref{thm:one} (ii).

\subsection{Proof of Theorem~\ref{thm:one} (iii)} Let $u$ be an even integer, and let $t=u/2$ in this subsection. Consider the following matrix
\begin{equation*}
    X=\left(\begin{array}{cccc}
        \sin(\pi/u) & \sin(3\pi/u) & \cdots & \sin((2t-1)\pi/u) \\
        \sin(2\pi/u) & \sin(6\pi/u) & \cdots & \sin(2(2t-1)\pi/u) \\
        \cdots&\cdots&\cdots&\cdots \\
        \sin (t\pi /u) & \sin (3t\pi/u) & \cdots & \sin (t(2t-1)\pi/u)
    \end{array}\right).
\end{equation*}
By Lemma~\ref{lem:matrix}, there is $\bm{\alpha}\in\RR^{t}$, with $|\alpha_i|\leq t^t$, such that $X\bm{\alpha}^{T}=(-1,\dots,0,0)^T$. 

For each odd integer $v$ ranging from the interval $[1,u)$, define $\mathcal{P}_{v}$ to be an infinite collection of pairs $(k_v,\l_v)$ of positive integers such that $k_v=(u+v)\l_v/(u-v)$. Let $F_{(k,\l)}(x)$ and $G_{(k,\l)}(x)$ be the function constructed in (\ref{eq:Fx}) and (\ref{eq:Gx}) with respect to the pair $(k,\l)$. Note that in the current constructions, for every $(k_1,\l_1), (k_2,\l_2)\in \mathcal{P}_v$,
we have
\[
F_{(k_1,\l_1)} \Big(\frac{x}{k_1-\l_1}\Big)=F_{(k_2,\l_2)} \Big(\frac{x}{k_2-\l_2}\Big),
\]
and we denote the above function by $F_v(x)$ since it only depends on $v$. Similarly, we also have 
\[
(k_1-\l_1)G_{(k_1,\l_1)} \Big(\frac{x}{k_1-\l_1}\Big)=(k_2-\l_2)G_{(k_2,\l_2)} \Big(\frac{x}{k_2-\l_2}\Big),
\]
and we denote the above function by $G_v(x)$.

Let $\bm{F}(x)=\left(F_1(x),F_3(x),\dots F_{2t-1}(x)\right),$ and  we construct
\[
F(x):=\bm{F}(x)\bm{\alpha}^{T}=\frac{2}{\pi}\sum_{n\geq1}\frac{\Phi(n)}{n}\cos(2\pi nx),
\]
where
\[
\Phi(n):=\begin{cases}
1&\text{ when } n\equiv 1,u-1\pmod{2u},\\
-1&\text{ when } n\equiv u+1,-1\pmod{2u},\\
0&\text{ otherwise, }
\end{cases}
\]
since in this case $k_v-\l_v$ is always even for every $(k_v,\l_v)$ in $\mathcal{P}_v$. Note that $\Phi(n)$ is a multiplicative function. Let $\MM$ be the set of square-free integers that only contains prime factors at most $P$ and $1\in\MM$, hence we have
\[
\sum_{m\in\MM}\frac{\Phi(m)\mu(m)}{m}F(mx)=\frac{2}{\pi}\Big(\cos(2\pi x)+\sum_{n\in\N}\frac{\Phi(n)}{n}\cos(2\pi nx)\Big).
\]

Similarly, we consider the matrix
\begin{equation*}
    Y=\left(\begin{array}{cccc}
        \sin(\pi/2u) & \sin(3\pi/2u) & \cdots & \sin((2t-1)\pi/2u) \\
        \sin(3\pi/2u) & \sin(9\pi/2u) & \cdots & \sin(3(2t-1)\pi/2u) \\
        \cdots&\cdots&\cdots&\cdots \\
        \sin ((2t-1)\pi/2u) & \sin (3(2t-1)\pi/2u) & \cdots & \sin ((2t-1)^2\pi/2u)
    \end{array}\right).
\end{equation*}
By Lemma~\ref{lem:matrix}, there is $\bm{\beta}\in\RR^{t}$, with $|\beta_i|\leq t^{t}$, such that $Y\bm{\beta}^{T}=(1,\dots,0,0)^T$. Let $\bm{G}(x)=\left(G_1(x),G_3(x),\dots,G_{2t-1}(x)\right),$ and  we construct
\[
G(x):=\bm{G}(x)\bm{\beta}^{T}=\frac{2}{\pi}\sum_{n\geq1}\frac{\Psi(n)}{n}\sin(2\pi nx),
\]
where
\[
\Psi(n):=\begin{cases}
1&\text{ when } n\equiv \pm1\pmod{4u},\\
-1&\text{ when } n\equiv \pm (2u-1)\pmod{4u},\\
0&\text{ otherwise. }
\end{cases}
\]
We also have $\Psi(n)$ is a multiplicative function. Hence 
\[
\sum_{m\in\MM}\frac{\Psi(m)\mu(m)}{m}G(mx)=\frac{2}{\pi}\Big(\sin(2\pi x)+\sum_{n\in\N}\frac{\Psi(n)}{n}\sin(2\pi nx)\Big).
\]

Finally, we apply Corollary~\ref{cor: key}. Using a similar computation employed in Section 5.1, we obtain that
\[
\max\Big\{\Big\Vert \sum_{m\in A}F(mx)\Big\Vert_{L^1(\TT)}, \Big\Vert \sum_{m\in A}G(mx)\Big\Vert_{L^1(\TT)}\Big\}\gg\frac{\log N}{\log\log N}. 
\]
This implies there is an odd integer $v\in [1,u)$ such that 
\[
\max\Big\{\Big\Vert \sum_{m\in A}F_v(mx)\Big\Vert_{L^1(\TT)}, \Big\Vert \sum_{m\in A}G_v(mx)\Big\Vert_{L^1(\TT)}\Big\}\gg\frac{\log N}{\log\log N}. 
\]
Therefore, from a similar argument we used in Section 5.1, we can conclude that the size of the maximal $(k_v,\l_v)$-sum-free subset of $A$ for every $(k_v,\l_v)\in\mathcal{P}_v$ is at least
\[
\frac{N}{k_v+\l_v}+c\frac{\log N}{\log\log N}
\]
for some positive $c$. This proves Theorem~\ref{thm:one} (iii).


\section{Structure of infinite $(k,\l)$-sum-free sets}
Given $A\subseteq \NN^{>0}$, the \emph{upper density} of $A$ is defined as
\[
\overline{d}(A)=\limsup_{N\to\infty}\frac{|A\cap[N]|}{N}.
\]
We also define the \emph{upper density on multiples} of $A$ by
\[
\widetilde{d}(A)=\limsup_{N\to\infty}\limsup_{n\to\infty}\frac{|A\cap (N!\cdot [n])|}{n}.
\]
In this section, we will prove the following theorem, which will be used in the next section when constructing the upper bound estimate for Theorem 1 (iv).
\begin{theorem}\label{thm:periodic}
Suppose that $A\subseteq\NN^{>0}$, and $A$ is $(k,\l)$-sum-free. Then $\widetilde{d}(A)\leq \frac{1}{k+\l}$.
\end{theorem}

We break the proof of this theorem into three lemmas. The first lemma says that if a $(k,\l)$-sum-free set $A$ contains a certain long arithmetic progression, then the upper density of $A$ is bounded.

\begin{lemma}\label{lem:5.1}
Let $A\subseteq \NN^{>0}$ be a $(k,\l)$-sum-free set. Let $x,s,d,m$ be positive integers, such that $s\in \l A-(k-1) A$, $x+d\cdot [m]\subseteq A$, and $s$ is in the coset $x+d\cdot\ZZ$. Then
\[
\overline{d}(A)\leq\frac{m+k+\l-2}{(k+\l)m+2(k+\l-2)}.
\]
\end{lemma}
\begin{proof}
Since $s\in\l A-(k-1)A$ and $A$ is $(k,\l)$-sum-free, we have $s\notin A$. We will only consider $s\leq x$, and the case when $s\geq x+m$ follows from the same proof. 
Since $x+d\cdot [m]\subseteq A$, then $\big(x+d\cdot [m]\big)\cap \big(\l A-(k-1)A\big)=\varnothing.$ Thus, there is $s_0\in x+d\cdot \ZZ$, such that $s_0\in\l A-(k-1)A$, and 
\begin{equation}\label{eq:s0}
\big(s_0+d\cdot [m]\big)\cap \big(\l A-(k-1)A\big)=\varnothing. 
\end{equation}
Let $s_0=\sum_{i=1}^\l a_i-\sum_{j=1}^{k-1}b_j$, where $a_i,b_j\in A$ for every $1\leq i\leq \l$ and $1\leq j\leq k-1$.

Let $B\subseteq A$ such that
\[
B:=\{b\in A\mid \big(b+d\cdot[m]\big)\cap A\neq\varnothing\}.
\]
Set $a_0=b_0=0$. Given integers $1\leq u\leq k-1$ and $2\leq v \leq \l$, let $$\C(u)=B+\sum_{j=1}^{k-u}b_j+(u-1)a_1,\quad\quad \D(v)=B+\sum_{j=0}^{\l-v}a_j+\sum_{i=0}^{v-1}b_i,$$
and $\C(k)=A+(k-1)a_1$, $\D(1)=A+\sum_{j=1}^{\l-1}a_j$.
Let $\mathscr{F}=\{\C(u),\D(v)\mid u\in[k],v\in [\l]\}$ be the collection of all $\C(u)$ and $\D(v)$.
\begin{claim}\label{cm:1}
Elements in $\mathscr{F}$ are pairwise disjoint.
\medskip

\noindent\emph{Proof of Claim \ref{cm:1}.}
Observe that for every $u\in[k]$ and $v\in[\l]$, $\C(u)\cap \D(v)=\varnothing$. Otherwise, we will get $kA\cap \l A\neq\varnothing$, contradicts that $A$ is $(k,\l)$-sum-free. Let $u_1,u_2\in[k]$ and $u_1< u_2$. Suppose that $\C(u_1)\cap \C(u_2)\neq\varnothing$. Then there exist $y_1\in B$ and $y_2\in A$, such that
\[
y_1+\sum_{j=k-u_2+1}^{k-u_1} b_j=y_2+(u_2-u_1)a_1.
\]
Then
\begin{align*}
s_0&=\sum_{i=1}^\l a_i-\sum_{j=1}^{k-1}b_j\\
&=y_1+\sum_{i=2}^\l a_i-y_2-(u_2-u_1-1)a_1-\sum_{j\in [1,k-u_2]\cup[k-u_1+1,k-1]}b_j.
\end{align*}
Since $y_1\in B$, thus there is $r\in [m]$ such that $y_1+rd\in A$. This implies $s_0+rd\in \l A-(k-1)A$, contradicts (\ref{eq:s0}).

Suppose $\D(v_1)\cap \D(v_2)\neq\varnothing$ for some $v_1,v_2\in[\l]$ and $v_1<v_2$. Similarly, there exist $y_1\in A$ and $y_2\in B$, such that
\[
y_1+\sum_{j=\l-v_2+1}^{\l-v_1}a_j=y_2+\sum_{i=v_1}^{v_2-1}b_i.
\]
Let $c_0=0$, and let $c_1,\dots,c_{v_2-v_1-1}\in A$ if $v_2>v_1+1$. Therefore
\begin{align*}
    s_0=y_2+\sum_{j\in[0,\l-v_2]\cup[\l-v_1+1,\l]}a_j+\sum_{t=0}^{v_2-v_1-1}c_t-y_1-\sum_{i\in[0,v_1-1]\cup[v_2,k-1]}b_i-\sum_{t=0}^{v_2-v_1-1}c_t.
\end{align*}
Observe $y_2\in B$ implies that there is $r\in[m]$, such that $y_2+rd\in A$. Hence $s_0+rd\in \l A-(k-1)A$, which contradicts (\ref{eq:s0}). 
\end{claim}

By Claim \ref{cm:1}, we obtain
\begin{equation}\label{eq:B1}
  (k+\l-2)\overline{d}(B)+2\overline{d}(A)\leq 1.  
\end{equation}
On the other hand, let $\N(t)=A\setminus B+td$ for every $t\in[m]$, and let $$\mathscr{G}=\Big\{A,A-(k-1)x+\sum_{i=1}^{\l-1}a_i,\N(t)\ \Big|\ t\in[m]\Big\}.$$
\begin{claim}\label{cm:2}
Elements in $\mathscr{G}$ are pairwise disjoint.
\medskip

\noindent\emph{Proof of Claim \ref{cm:2}.}
Suppose there are $u,v\in [m]$, $u<v$, such that $\N(u)\cap \N(v)\neq\varnothing$. Thus we have $c\in A\setminus B$ such that $c_1+(u-v)d\in A$, and this contradicts the assumption of $B$. Same conclusion holds if $A\cap \N(u)\neq\varnothing$. Observe that if $A\cap (A-(k-1)x+\sum_{i=1}^{\l-1}a_i)\neq\varnothing$, it will contradict that $A$ is $(k,\l)$-sum-free. Finally, we assume that there are $c_1,c_2\in A$, $u\in [m]$ such that
\[
c_1+ud=c_2-(k-1)x+\sum_{i=1}^{\l-1}a_i.
\]
Thus, $c_1+x+ud+(k-2)x=c_2+\sum_{i=1}^{\l-1}a_i$. Since $x+d\cdot[m]\subseteq A$, this contradicts $A$ is $(k,\l)$-sum-free.
\end{claim}

Thus, by Claim~\ref{cm:2}, we get
\[
(m+2)\overline{d}(A)-m\overline{d}(B)\leq 1.
\]
Together with (\ref{eq:B1}), this finishes the proof.
\end{proof}

The next lemma is a finite version of the Szemer\'edi Theorem~\cite{Sz}, and we will use it to find the arithmetic progression in Lemma~\ref{lem:5.1}.
\begin{lemma}[\cite{Sz}]\label{lem:Sz}
For every $\varepsilon>0$ and $m\in\NN^{>0}$, there is $L=L(\varepsilon,m)>0$ such that every set $A\subseteq\mathbb{N}^{>0}$ with $\overline{d}(A)>\varepsilon$, there exist $x\in\NN$, $d<L$, and $x+d\cdot [m]\subseteq A$.
\end{lemma}

Our final lemma says that a $(k,\l)$-sum-free set $A$ with large upper density should be periodic. This structural result can be viewed as a generalization of the \L uczak--Schoen Theorem~\cite{LS97}.
\begin{lemma}\label{lem:5.2}
Let $\varepsilon>0$. Then there is $D>0$ such that the following holds. Let $A\subseteq\NN^{>0}$ be a $(k,\l)$-sum-free set, and $\overline{d}(A)>\frac{1}{k+\l}+\varepsilon$. Then $A$ is contained in a periodic $(k,\l)$-sum-free set with period $D$.
\end{lemma}
\begin{proof}
We pick $m\in\NN^{>0}$ such that
\begin{equation}\label{eq:density}
\frac{m+k+\l-2}{(k+\l)m+2(k+\l-2)}<\frac{1}{k+\l}+\varepsilon.
\end{equation}
Let $L=L(\varepsilon,m)$ be as in Lemma~\ref{lem:Sz}.
Let $D=L!$. Suppose the lemma fails.  Let $C\subseteq\NN^{>0}$ be a periodic set with period $D$, consists of all positive integers in every coset $a+D\cdot\ZZ$ for $a\in A$. Thus $C$ is not $(k,\l)$-sum-free. This means, there are $a_1,\dots,a_\l$ and $b_1,\dots,b_k$ in $C$ such that $\sum_{i=1}^\l a_i=\sum_{j=1}^k b_j$. Let $P$ be the ``$(k,\l)$-sum-free'' part of $C$. That is, 
\[
P=C\setminus \big(\l C-(k-1)C\big).
\]
Set $a_0=b_0=0$. For every $u\in[k]$ and $v\in[\l]$, let
\[
\MM(u)=P+\sum_{j=0}^{k-u}b_j+(u-1)a_1,\qquad \N(v)=P+\sum_{i=0}^{\l-v}a_i+(v-1)b_1.
\]
Let $\mathscr{F}$ be the collection of all $\MM(u)$ and $\N(v)$. 
\begin{claim}\label{cm:3}
Elements in $\mathscr{F}$ are pairwise disjoint.
\medskip

\noindent\emph{Proof of Claim \ref{cm:3}.}
Observe that for every $u\in [k]$ and $v\in[\l]$, $\MM(u)\cap \N(v)=\varnothing$. Otherwise there are $p_1,p_2\in P$, such that
\[
p_1=p_2+\sum_{i=0}^{\l-v}a_1+(v-1)b-\sum_{j=0}^{k-u}b_j-(u-1)a_1\in \l C-(k-1)C,
\]
contradicts the assumption of $P$. Now, suppose $u_1,u_2\in [k]$, $u_1<u_2$, such that $\MM(u_1)\cap \MM(u_2)\neq\varnothing$. The case that $\N(v_1)\cap \N(v_2)\neq\varnothing$ can be proved in the same way. Thus, there exist $p_1,p_2\in P$, such that
\[
p_1+\sum_{j=k-u_2+1}^{k-u_1}b_j=p_2+(u_2-u_1)a_1.
\]
This implies
\[
0=\sum_{j=1}^k b_j-\sum_{i=1}^\l a_i=p_2+(u_2-u_1-1)a_1+\sum_{j\in[0,k-u_2]\cup [k-u_1+1,k]}b_j-\sum_{i=2}^\l a_i-p_1,
\]
hence $P\cap(\l C-(k-1)C)\neq\varnothing$, contradiction.
\end{claim}

By Claim \ref{cm:3}, we obtain that $\overline{d}(P)\leq\frac{1}{k+\l}$. This means, $\overline{d}(A\setminus P)\geq\varepsilon$. By Lemma~\ref{lem:Sz}, $A\setminus P$ contains a progression $x+d\cdot [m]$, and $d<L$. By the way we construct $P$, there are $s_1,\dots,s_\l$ and $t_1,\dots,t_{k-1}$ in $C$ such that 
\[
x+dm=\sum_{i=1}^\l s_i-\sum_{j=1}^{k-1}t_j.
\]
Hence there are $e_1,\dots,e_\l$ and $f_1,\dots,f_{k-1}$ in $A$, such that for every $i\in[\l]$ and $j\in[k-1]$, we have that $e_i\in s_i+D\cdot \ZZ$, and $f_j\in t_j+D\cdot \ZZ$. Let $s=\sum_{i=1}^\l e_i-\sum_{j=1}^{k-1}f_j$, thus $s\in \l A-(k-1)A$, and $s\in x+D\cdot \ZZ$. Since $d\mid D$, we have $s\in x+d\cdot \ZZ$. By Lemma~\ref{lem:5.1}, we have that 
\[
\overline{d}(A)\leq\frac{m+k+\l-2}{(k+\l)m+2(k+\l-2)},
\]
and this contradicts (\ref{eq:density}).
\end{proof}

Now we can prove the main result of this section.
\begin{proof}[Proof of Theorem~\ref{thm:periodic}]
Let $A/N!:=\{a\mid aN!\in A\}$. Thus $\widetilde{d}(A)>0$ implies that $A/N!$ contains a multiple of every natural number. In particular, $A/N!$ is not contained in a periodic $(k,\l)$-sum-free set. By Lemma~\ref{lem:5.2}, $\overline{d}(A/N!)\leq\frac{1}{k+\l}$.
Observe that $\widetilde{d}(A)=\limsup_{N\to\infty}\overline{d}(A/N!)$, thus $\widetilde{d}(A)\leq\frac{1}{k+\l}$.
\end{proof}

\section{Upper bound constructions}

Recall a \emph{F\o lner sequence} in $(\NN,\cdot)$ is any sequence $\Phi:m\mapsto\Phi_m$ of finite non-empty subsets of $\NN$, such that for every $a\in\NN$,
\[
\lim_{m\to\infty}\frac{|\Phi_m\triangle (a\cdot \Phi_m)|}{|\Phi_m|}=0.
\]

F\o lner sequence has been used as some good constructions in many additive combinatorics problems, see \cite{A,B} for example.  In this section, we will show that the sets in F\o lner sequence will never have large $(k,\l)$-sum-free subsets. In fact, we will prove the following theorem.
\begin{theorem}\label{thm:folner}
Let $\Phi=\{\Phi_m\}$ be a F\o lner sequence in $(\NN,\cdot)$. Suppose there are infinitely many $m$ such that $\Phi_m$ has a $(k,\l)$-sum-free set of size at least $\delta|\Phi_m|$ for some positive real number $\delta\leq 1$. Then there exists a $(k,\l)$-sum-free set $A\subseteq\NN$ such that $\widetilde{d}(A)\geq\delta$.
\end{theorem}

Theorem~\ref{thm:one} (iv) follows easily from Theorem~\ref{thm:folner} and Theorem~\ref{thm:periodic}.
\begin{proof}[Proof of Theorem~\ref{thm:folner}]
By passing to a subsequence, we may assume for every $\Phi_m\in\Phi$, there is a $(k,\l)$-sum-free set $\phi_m\subseteq \Phi_m$, such that $|\phi_m|/|\Phi_m|\geq\delta$.
Let $\beta\NN$ be the collection of ultrafilters, and let $\mathscr{U}\in\beta\NN\setminus\NN$ be a non-principal ultrafilter. Let $^{*}\ZZ=\prod_{m\to\mathscr{U}}\ZZ$ be the ultrapower of $\ZZ$. Let $\Sigma$ be the Loeb $\sigma$-algebra on $^{*}\ZZ$. Let $\mu_L$ be the Loeb measure induced by $\mu_m$, where $\mu_m(X)=|X\cap \Phi_m|/|\Phi_m|$ for every $X\subseteq\ZZ$. Let $\phi=\prod_{m\to\mathscr{U}}\phi_m$ be the internal set. Then by \L o\'s's Theorem, $\phi$ is $(k,\l)$-sum-free, and
\[
\mu_L(\phi)=\mathrm{st}\left(\lim_{m\to\mathscr{U}}\mu_m(\phi_m)\right)\geq\delta.
\]
\begin{claim}\label{cm:4}
For every $a\in\NN$, the map $x\mapsto ax$ is $\Sigma$-measurable and $\mu_L$-preserving.
\medskip

\noindent\emph{Proof of Claim \ref{cm:4}.}
Note that $x\mapsto ax$ sends internal sets to internal sets, thus it is $\Sigma$-measurable. For every $X\subseteq\ZZ$, since 
\[
\mu_m(X)-\mu_m(a\cdot X)=\frac{|X\cap \Phi_m|-|(a\cdot X)\cap \Phi_m|}{|\Phi_m|}\leq \frac{|(a\cdot \Phi_m)\triangle \Phi_m|}{|\Phi_m|}\to 0
\]
as $m\to \infty$, it preserves the Loeb measure $\mu_L$.
\end{claim}

Now we are able to apply the probabilistic argument used in the proof of Proposition~\ref{prop:T1} on the set $\phi$. For every $x\in {^{*}\ZZ}\setminus\{0\}$, let $A_x:=\{a\in\NN\mid ax\in \phi\}$. Thus $A_x$ is $(k,\l)$-sum-free. By Claim \ref{cm:4}, $\widetilde{d}(A_x)$ is $\Sigma$-measurable on $x$. Suppose $x$ is chosen uniformly at random with respect to the measure $\mu_L$. By Fatou's Lemma,
\begin{align*}
    \mathbb{E}(\widetilde{d}(A_x))&\geq\limsup_{N\to\infty}\limsup_{n\to\infty} \mathbb{E}\left(\frac{|A_x\cap(N!\cdot[n])|}{n}\right)\\
    &=\limsup_{N\to\infty}\limsup_{n\to\infty}\frac{1}{n}\sum_{j=1}^n\mathbb{P}(jN!x\in \phi).
\end{align*}
By Claim \ref{cm:4}, we have
\[
\mathbb{E}(\widetilde{d}(A_x))\geq \limsup_{N\to\infty}\limsup_{n\to\infty}\frac{1}{n}\sum_{j=1}^n\mathbb{P}(x\in \phi)=\mu_L(\phi)\geq\delta.
\]
Thus by Pigeonhole Principle, there exists a set $A_x\subseteq\NN$ for some $x\in{^*\ZZ}\setminus\{0\}$ such that $\widetilde{d}(A_x)\geq\delta$.
\end{proof}

\section{Restricted $(k,\l)$-sum-free sets}

In this section, we prove Theorem~\ref{thm:1.2}. Since restricted $(k,\l)$-sum-free can be expressed by first order formula, once we prove the conclusion in Theorem~\ref{thm:periodic} also works for restricted $(k,\l)$-sum-free sets, Theorem \ref{thm:1.2} follows by using the same proof in Theorem \ref{thm:folner}. More precisely, in the proof of Theorem \ref{thm:folner}, if $A_x=\{a\in\NN\mid ax\in \phi\}$ is not restricted $(k,\l)$-sum-free for some $x\in{^*\ZZ}\setminus\{0\}$, since the map $a\mapsto ax$ is injective, we also have that $\phi$ is not restricted $(k,\l)$-sum-free.

We first consider the analogue of Lemma~\ref{lem:5.1} for restricted $(k,\l)$-sum-free sets. The similar argument also works here, with a different and more involved constructions of sets $\C(u)$, $\D(v)$, and $\N(t)$, and a more careful analysis. These new constructions will lead a slightly different structure for the large infinite restricted $(k,\l)$-sum-free sets in Lemma~\ref{lem:7.2}, compared to the non-restricted setting.
\begin{lemma}\label{lem:81}
Let $k,\l$ be positive integers, and $\l<k\leq2\l+1$. Suppose $A\subseteq \NN^{>0}$ be a restricted $(k,\l)$-sum-free set. Define $W\subseteq\NN^{>0}$, satisfies that for every $w\in W$, there are $\l$ distinct elements $y_1,\dots,y_\l\in A$, and $k-1$ distinct elements $z_1,\dots,z_{k-1}\in A$, such that $w\neq z_i$ for $i\in[k-1]$, and $w=\sum_{j=1}^\l y_j-\sum_{i=1}^{k-1}z_i$. Let $x,s,d,m$ be integers, such that $s\in W$, $m>k+\l$, $x+d\cdot [m]\subseteq A$, and $s$ is in the coset $x+d\cdot\ZZ$. Then
\[
\overline{d}(A)\leq\frac{m-2}{(k+\l)(m-k-\l)+2(k+\l-2)}.
\]
\end{lemma}
\begin{proof}
$s\in W$ implies that $s\notin A$ since $A$ is restricted $(k,\l)$-sum-free. We only consider the case when $s<x$.  Since $A\cap W=\varnothing$, there is $s_0\in x+d\cdot\ZZ$ such that $s_0\in W$ and $(s_0+d\cdot[m])\cap W=\varnothing$. Thus there are $\l$ distinct elements $a_1,\dots,a_l\in A$, and $k-1$ distinct elements $b_1,\dots,b_{k-1}\in A$, $s_0\neq b_j$ for every $j\in[k-1]$, and $s_0=\sum_{i=1}^\l a_i-\sum_{j=1}^{k-1}b_j$. Let $\mathcal{E}$ consists of $k-1$ distinct elements $e_1,\dots,e_{k-1}\in A$, and all of them are disjoint from $\{a_i\}_{i=1}^\l$, $\{b_j\}_{j=1}^{k-1}$, $\{s_0\}$ and $s_0+d\cdot[m]$. Let
\begin{equation}\label{eq:A'}
A'=A\setminus\Big(\bigcup_{i=1}^\l\{a_i\}\cup\bigcup_{j=1}^{k-1}\{b_j\}\cup \mathcal{E}\cup\{s_0\}\cup (s_0+d\cdot [m])\Big).
\end{equation}
Observe that 
\begin{equation}\label{eq:b_j}
    (s_0+d\cdot[m])\cap\{b_j\}_{j=1}^{k-1}=\varnothing,
\end{equation}since $b_j\in W$ for every $j\in[k-1]$. Let $m'=m-k-\l$, we claim that
\begin{equation}\label{eq:a_i}
     (s_0+d\cdot[m'])\cap\{a_i\}_{i=1}^{\l}=\varnothing.
\end{equation}
Otherwise, suppose there is $r\in[m']$ such that $s_0+rd=a_t$ for some $t\in[\l]$. Then
\[
x'+\sum_{j=1}^{k-1}b_j=x'+rd+\sum_{j=1,j\neq t}^\l a_j.
\]
By taking $x'\in x+d\cdot[0,m-r]$, then both $x'$ and $x'+rd$ are in $A$. Since $m-r\geq k+\l$, there is $\alpha\in[0,m-r]$ such that $x+\alpha d\notin\{b_j\}_{j=1}^{k-1}$, and $x+(\alpha+r)d\notin \{a_i\}_{i=1}^{\l}$. This contradicts that $A$ is restricted $(k,\l)$-sum-free.

Let $B=\{b\in A'\mid (b+d\cdot[m'])\cap A\neq\varnothing\},$
and let
\[
B'=B\setminus\left(\Big(\bigcup_{i=1}^\l\{a_i\}\cup\mathcal{E}\Big)-d\cdot[m']\right).
\]
Let $c_0=0$, $c_i=a_i$ when $i\in[\l]$, and $c_j=a_{j-\l}$ when $j\in[\l+1,k-1]$. For $u\in[k-1]$ and $v\in [2,\l]$, let
\[
\C(u)=B'+\sum_{j=1}^{k-u}b_j+\sum_{i=0}^{u-1}c_i,\quad \D(v)=B'+\sum_{i=0}^{\l-v}a_i+\sum_{j=0}^{v-1}b_j,
\]
and $\C(k)=A'+\sum_{i=0}^{k-1}c_i$, $\D(1)=A'+\sum_{i=1}^{\l-1}a_i$. Let $\mathscr{F}$ consists of all $\C(u)$ and $\D(v)$, then Claim \ref{cm:1} still holds. In fact, suppose there are $u_1,u_2\in [k]$, $u_1<u_2$ such that $\C(u_1)\cap \C(u_2)\neq\varnothing$ (the case when $\D(v_1)\cap \D(v_2)\neq\varnothing$ is simpler). Then there exist $y_1\in B'$, $y_2\in A'$ such that
\[
y_1+\sum_{j=k-u_2+1}^{k-u_1}b_j=y_2+\sum_{i=u_1}^{u_2-1}c_i.
\]
Let $e_0=0$, and $e_1,\dots,e_{u_2-u_1-1}\in \mathcal{E}$ if $u_2>u_1+1$. If $u_2\leq\l$, we have
\[
s_0=y_1+\sum_{i\in[0,u_1-1]\cup[u_2,\l]}a_i+\sum_{t=0}^{u_2-u_1-1}e_t-y_2-\sum_{j\in[0,k-u_2]\cup[k-u_1+1,k-1]}b_j-\sum_{t=0}^{u_2-u_1-1}e_t.
\]
If $u_1\geq \l+1$, we get
\[
s_0=y_1+\sum_{i\in[0,u_1-1-\l]\cup[u_2-\l,\l]}a_i+\sum_{t=0}^{u_2-u_1-1}e_t-y_2-\sum_{j\in[0,k-u_2]\cup[k-u_1+1,k-1]}b_j-\sum_{t=0}^{u_2-u_1-1}e_t.
\]
If $u_1\leq\l$, $u_2\geq\l+1$, and $u_2-u_1+1\leq \l$,
\[
s_0=y_1+\sum_{i\in[u_2-\l,u_1-1]}a_i+\sum_{t=0}^{u_2-u_1-1}e_t-y_2-\sum_{j\in[0,k-u_2]\cup[k-u_1+1,k-1]}b_j-\sum_{t=0}^{u_2-u_1-1}e_t.
\]
If $u_1\leq\l$, $u_2\geq\l+1$, and $u_2-u_1\geq \l$. Let $e_0=0$, $e_1,\dots,e_{\l-1}\in \mathcal{E}$ if $\l>1$. Thus
\[
s_0=y_1+\sum_{t=0}^{\l-1}e_t-y_2-\sum_{j\in[0,k-u_2]\cup[k-u_1+1,k-1]}b_j-\sum_{t=0}^{\l-1}e_t-\sum_{i=u_1}^{u_2-1-\l}a_i.
\]
Note that $k\leq 2\l+1$ implies $u_2-1-\l\leq\l$. 

In any case, since $y_1\in B$, by (\ref{eq:A'}), (\ref{eq:b_j}), and (\ref{eq:a_i}), there is $r\in[m']$ such that $s_0+rd\in W$, which contradicts the assumption of $s_0$. Therefore,
\begin{equation}\label{eq:7.1.1}
  (k+\l-2)\overline{d}(B)+2\overline{d}(A)\leq 1, 
\end{equation}
since $\overline{d}(A')=\overline{d}(A)$ and $\overline{d}(B')=\overline{d}(B)$.

We also modify the construction of $\N(t)$ in a similar way. For every $t\in[m']$, let $\N(t)=A'\setminus B+td$. Let $e_0=0$, and $e_1,\dots,e_{k-2}\in \mathcal{E}$ if $k\geq3$. Let $A''=A'\setminus(x+d\cdot[m'])$. Define
\[
\mathscr{G}=\Big\{\N(t),A',A''+\sum_{i=1}^{\l-1}a_i-x-\sum_{j=0}^{k-2}e_j\ \Big|\  t\in[m']\Big\}
\]
Then by using the similar argument, it is easy to see that Claim \ref{cm:2} still holds. We omit the details here. We have
\[
(m-k-\l+2)\overline{d}(A)-(m-k-\l)\overline{d}(B)\leq 1,
\]
since $\overline{d}(A'')=\overline{d}(A)$. Together with (\ref{eq:7.1.1}), finishes the proof.
\end{proof}

Next, we consider the analogue of Lemma~\ref{lem:5.2} for restricted $(k,\l)$-sum-free sets. The structure here is slightly different from the $(k,\l)$-sum-free sets.
\begin{lemma}\label{lem:7.2}
Let $\varepsilon>0$ and let $k,\l$ be positive integers with $\l<k\leq2\l+1$. Then there is $D>0$ such that the following holds. Let $A\subseteq\NN^{>0}$ be a restricted $(k,\l)$-sum-free set, and $\overline{d}(A)>\frac{1}{k+\l}+\varepsilon$. Then after removing at most $D(2k+\l)$ elements from $A$, it is contained in a periodic restricted $(k,\l)$-sum-free set with period $D$.
\end{lemma}
\begin{proof}
We pick $m>k+\l$ such that
\begin{equation}\label{eq:7.2.1}
\frac{m-2}{(k+\l)(m-k-\l)+2(k+\l-2)}<\frac{1}{k+\l}+\varepsilon.
\end{equation}
Let $L=L(\varepsilon,m)$ be as in Lemma~\ref{lem:Sz}, and let $D=L!$. We consider the partition of $\NN$ into cosets:
\[
\NN=\bigcup_{x\in[D]}x+D\cdot\NN.
\]
For every $x\in[D]$, let $\NN_x=x+D\cdot\NN$, and $A_x=A\cap \NN_x$. Let $A'$ be a subset of $A$, obtained by removing $A_x$ from $A$ when $|A_x|<2k+\l$. Hence $\overline{d}(A')=\overline{d}(A)$. Next, we are going to show that $A'$ is contained in a periodic restricted $(k,\l)$-sum-free set with period $D$. Suppose this is not the case. Let
\[
C=\Big(\bigcup_{a\in A'} a+D\cdot\ZZ\Big)\cap\NN^{>0}.
\]
Thus $C$ is not restricted $(k,\l)$-sum-free. This means, there are $\l$ distinct elements $a_1,\dots,a_\l\in C$ and $k$ distinct elements $b_1,\dots,b_k\in C$, such that $\sum_{i=1}^\l a_i=\sum_{j=1}^k b_j$. Let $P$ be the ``$(k,\l)$-sum-free'' part of $C$, that for every $w\in P$, every $k-1$ distinct elements $y_1,\dots,y_{k-1}\in C\setminus\{w\}$, and every $\l$ distinct elements $z_1,\dots,z_\l\in C$, we have $w+\sum_{i=1}^{k-1}y_i\neq\sum_{j=1}^\l z_j$. Let $e_0=0$, and let $\mathcal{E}$ consists of $k-1$ distinct elements $e_1,\dots,e_{k-1}\in C$, such that $\mathcal{E}$ is disjoint from $\{a_i\}_{i=1}^\l$ and $\{b_j\}_{j=1}^k$. 
\[
P'=P\setminus\Big(\bigcup_{i=1}^\l\{a_i\}\cup\bigcup_{j=1}^k\{b_j\}\cup\mathcal{E}\Big).
\]

Set $a_0=b_0=c_0=0$. Let $c_t=a_t$ when $t\in[\l]$, and $c_t=a_{t-\l}$ when $t\in[\l+1,k-1]$. For every $u\in[k]$ and $v\in[\l]$, let
\[
\MM(u)=P'+\sum_{j=0}^{k-u}b_j+\sum_{t=0}^{u-1}c_t,\qquad \N(v)=P'+\sum_{i=0}^{\l-v}a_i+\sum_{t=0}^{v-1}b_t.
\]
Let $\mathscr{F}$ be the collection of all $\MM(u)$ and $\N(v)$. Then elements in $\mathscr{F}$ are pairwise disjoint. Otherwise, suppose there are $u_1,u_2\in[k]$, $u_1<u_2$ such that $\MM(u_1)\cap \MM(u_2)\neq\varnothing$ (the case when $\N(v_1)\cap \N(v_2)\neq\varnothing$ is simpler). Thus, there are $y_1,y_2\in P'$, such that
\[
y_1+\sum_{k-u_2+1}^{k-u_1}b_j=y_2+\sum_{t=u_1}^{u_2-1}c_t.
\]
Let $e_1,\dots,e_{u_2-u_1-1}\in\mathcal{E}$ if $u_2>u_1+1$. If $u_2\leq\l$, we have
\begin{align*}
  0&=\sum_{i=1}^\l a_i-\sum_{j=1}^k b_j\\
  &=y_1+\sum_{i\in[0,u_1-1]\cup[u_2,\l]}a_i+\sum_{t=0}^{u_2-u_1-1}e_t-y_2-\sum_{j\in[0,k-u_2]\cup[k-u_1+1,k]}b_j-\sum_{t=0}^{u_2-u_1-1}e_t. 
\end{align*}
If $u_1\geq\l+1$, we have
\[
0=y_1+\sum_{i\in[0,u_1-1-\l]\cup[u_2-\l,\l]}a_i+\sum_{t=0}^{u_2-u_1-1}e_t-y_2-\sum_{j\in[0,k-u_2]\cup[k-u_1+1,k]}b_j-\sum_{t=0}^{u_2-u_1-1}e_t. 
\]
If $u_1\leq\l$, $u_2\geq\l+1$, and $\l\geq u_2-u_1$, we get
\[
0=y_1+\sum_{i=u_2-\l}^{u_1-1}a_i+\sum_{t=0}^{u_2-u_1-1}e_t-y_2-\sum_{j\in[0,k-u_2]\cup[k-u_1+1,k]}b_j-\sum_{t=0}^{u_2-u_1-1}e_t. 
\]
If $u_1\leq\l$, $u_2\geq\l+1$, and $\l<u_2-u_1$. Let $e_1,\dots,e_{\l-1}\in\mathcal{E}$ if $\l>1$, we get
\[
0=y_1+\sum_{t=0}^{\l-1}e_t-y_2-\sum_{j\in[0,k-u_2]\cup[k-u_1+1,k]}b_j-\sum_{i=u_1}^{u_2-1-\l} a_i-\sum_{t=0}^{\l-1}e_t.
\]

In any case, we get a contradiction with the assumption of $P'$ and the fact that $y_2\in P'$.
Therefore, $$\overline{d}(P)\leq\frac{1}{k+\l},$$ since $\overline{d}(P')=\overline{d}(P)$. This means, $\overline{d}(A'\setminus P)\geq\varepsilon$. By Lemma~\ref{lem:Sz}, $A'\setminus P$ contains a progression $x+d\cdot [m]$, and $d<L$. By the way we construct $P$, there are $\l$ distinct elements $s_1,\dots,s_\l\in C$ and $k-1$ distinct elements $t_1,\dots,t_{k-1}$ in $C\setminus\{x+m\}$ such that 
\[
x+m=\sum_{i=1}^\l s_i-\sum_{j=1}^{k-1}t_j.
\]
By the way we construct $A'$, for every $r\in[D]$, if $|A'\cap\NN_r|>0$, then $|A'\cap\NN_r|\geq 2k+\l$. Thus, there are $\l$ distinct elements $\alpha_1,\dots,\alpha_\l\in A'$ and $k-1$ distinct elements $\beta_1,\dots,\beta_{k-1}\in A'$, such that for every $i\in[\l]$ and $j\in[k-1]$, we have that $\alpha_i\in s_i+D\cdot \ZZ$, and $\beta_j\in t_j+D\cdot \ZZ$. Let $s=\sum_{i=1}^\l \alpha_i-\sum_{j=1}^{k-1}\beta_j$. Note that $|A'\cap\NN_r|\geq 2k+\l$ also implies that there is $r'\in[\l]$, and $M\subseteq \NN^{>0}$, $|M|\geq k$, such that $$\alpha_{r'}+D\cdot M\subseteq A',\quad (\alpha_{r'}+D\cdot M)\cap\bigcup_{i=1}^{\l}\{\alpha_i\}=\varnothing.$$
Thus if $s\cap\{\beta_j\}_{j=1}^{k-1}\neq\varnothing$, then by changing $\alpha_{r'}$ by $\alpha_{r'}+nD$ for some $n\in M$, one can make $s+nD\cap\{\beta_j\}_{j=1}^{k-1}=\varnothing$. Since $d\mid D$, we have $s\in x+d\cdot \ZZ$. By Lemma~\ref{lem:5.1}, we have that 
\[
\overline{d}(A)\leq\frac{m-2}{(k+\l)(m-k-\l)+2(k+\l-2)},
\]
and this contradicts (\ref{eq:7.2.1}).
\end{proof}

Let $A$ be a restricted $(k,\l)$-sum-free set, and let $A'$ be a subset of $A$ obtained by removing finitely many elements from $A$. Observe that, if $A'$ is contained in a periodic restricted $(k,\l)$-sum-free set, then $A$ cannot contain a multiple of every natural number. Thus, using the same proof in Theorem~\ref{thm:periodic}, we conclude that $\widetilde{d}(A)\leq\frac{1}{k+\l}$ if $A$ is restricted $(k,\l)$-sum-free.

\section{Concluding Remarks}
In this paper, we first study $\M_{(k,\l)}(N)$. In particular, we prove that Conjecture~\ref{conj:kl} is true for infinitely many $(k,\l)$. While solving Conjecture~\ref{conj:kl}  might  not  be  a  realistic  target  at  the  moment, the following conjecture for the case when $k-\l\geq2$ might  be   feasible. This is because in this case, Lemma~\ref{lem:T} implies that we have two different asymmetric maximal $(k,\ell)$-sum-free open sets in $\TT$, and the technique developed in this paper might be useful. 
\begin{conjecture}
Let $k,\l$ be positive integers and $k\geq\l+2$. Then there is a function $\omega(N)\to\infty$ as $N\to\infty$, such that $$  \M_{(k,\l)}(N)\geq \frac{N}{k+\l}+\omega(N).$$
\end{conjecture}

We also study $\widehat{\M}_{(k,\l)}(N)$ in Theorem~\ref{thm:1.2}. As we can see in the proofs in Section~8, when $k> 2\l+1$, the current strategy failed to obtain disjoint sets $\mathcal{C}$ and $\mathcal{D}$ in the proof of Lemma~\ref{lem:81}, as well as disjoint sets $\MM$ and $\N$ in the proof of Lemma~\ref{lem:7.2}. Although we think it is very likely that the conclusion in Theorem~\ref{thm:1.2} holds for every $k$ and $\l$, the case $k>2\l+1$ may require some new ingredients. 

\begin{conjecture}
For every positive integers $k,\l$ with $k>2\l+1$,
$$\widehat{\M}_{(k,\l)}(N)=\Big(\frac{1}{k+\l}+o(1)\Big)N.$$
\end{conjecture}

A $(k,\l)$-sum-free set is a set forbidding a linear equation $\sum_{i=1}^\l x_i=\sum_{j=1}^k y_j$. Another interesting direction is to consider the analogue problem on sets forbidding a system of linear equations. One of the most interesting problems along this line might be forbidding the projective cubes. Given a multiset $S=\{s_1,\dots,s_d\}$, a \emph{$d$-dimensional projective cube} generated by $S$ is 
\[
\square^d (S):=\Big\{\sum_{i\in I}s_i\ \Big|\ \varnothing\neq I\subseteq [d]\Big\}.
\]
A set is \emph{$\square^d$-free} if it does not contain any $d$-dimensional projective cubes as its subsets. 
Extremal properties of projective cubes have a vast literature, see e.g. \cite{AF88,EF90,GR98,LW19}. The problem on forbidding $d$-dimensional projective cubes can be viewed as a generalization of sum-free sets in another direction, since a sum-free set is also a $\square^2$-free set. Thus, the following problem is worthwhile to pursue. 
\begin{question}
Let $d\geq3$ be an integer. Define
\[
\M_{\square^d}(N):=\inf_{\substack{A\subseteq\NN^{>0}\\|A|=N}}\max_{\substack{B\subseteq A\\B\text{ \emph{is} }\square^d\text{\emph{-free}}}}|B|.
\]
Determine $\M_{\square^d}(N)$.
\end{question}

\section*{Acknowledgements}

The authors would like to thank Bela Bajnok and Noah Kravitz for pointing out some missing references. They are also deeply indebted to the referee for carefully reading the manuscript and pointing out several mistakes made in the earlier versions.

\bibliographystyle{amsplain}
\bibliography{ref}

\end{document}